\documentclass{amsart}

\usepackage{lipsum}
\usepackage{amsfonts}
\usepackage{graphicx}
\usepackage{epstopdf}
\usepackage{algorithmic}
\usepackage{mathtools}
\usepackage{placeins}
\usepackage{amssymb}
\usepackage{enumitem}
\usepackage{xcolor}
\usepackage{url}
\usepackage{hyperref}
\usepackage{thm-restate}
\ifpdf
  \DeclareGraphicsExtensions{.eps,.pdf,.png,.jpg}
\else
  \DeclareGraphicsExtensions{.eps}
\fi

\usepackage{amsopn}

\newcommand{\R}{\mathbb{R}}

\newcommand{\N}{\mathbb{N}}

\newcommand{\Td}{\mathbb{T}^d}
\newcommand{\Rd}{\mathbb{R}^d}

\newcommand{\cB}{\mathcal{B}}

\newcommand{\cE}{\mathcal{E}}

\newcommand{\cI}{\mathcal{I}}

\newcommand{\cK}{\mathcal{K}}
\newcommand{\cL}{\mathcal{L}}
\newcommand{\cM}{\mathcal{M}}
\newcommand{\cN}{\mathcal{N}}

\newcommand{\cQ}{\mathcal{Q}}
\newcommand{\cR}{\mathcal{R}}

\newcommand{\cV}{\mathcal{V}}

\newcommand{\cX}{\mathcal{X}}
\newcommand{\cY}{\mathcal{Y}}

\newcommand{\dz}{\;\mathsf{d}z}
\newcommand{\dx}{\;\mathsf{d}x}
\newcommand{\dy}{\;\mathsf{d}y}

\newcommand{\sfK}{\mathsf{K}}

\newcommand{\Obar}{\overline{\Omega}}

\newcommand{\mfK}{\mathsf{K}}

\renewcommand{\tilde}{\widetilde}
\renewcommand{\hat}{\widehat}

\newcommand\p{\partial}
\newcommand{\la}{\langle}
\newcommand{\ra}{\rangle}

\renewcommand{\L}{{\lfloor \bl^\ast/2 \rfloor}}
\newcommand{\bl}{\bm{\ell}}

\definecolor{darkred}{rgb}{.7,0,0}
\definecolor{darkblue}{rgb}{0,0,.7}

\newtheorem{theorem}{Theorem}[section]
\newtheorem{lemma}[theorem]{Lemma}
\newtheorem{corollary}[theorem]{Corollary}
\newtheorem{proposition}[theorem]{Proposition}

\theoremstyle{definition}
\newtheorem{definition}[theorem]{Definition}

\theoremstyle{remark}
\newtheorem{remark}[theorem]{Remark}
\newtheorem{assumption}[theorem]{Assumption}

\numberwithin{equation}{section}
\AtBeginDocument{%
}

\begin{document}

\title{Enforcing Dirichlet Boundary Conditions in Operator Learning}

\author{Andrew M. Stuart}
\address{Department of Computing and Mathematical Sciences, Caltech, Pasadena, CA}
\email{astuart@caltech.edu}
\thanks{The first author was supported by the Department of Defense Vannevar Bush
Faculty Fellowship N00014-22-1-2790}

\author{Margaret Trautner}
\address{Department of Mathematics, ETH Z\"urich, Z\"urich, Switzerland}
\email{margaret.trautner@sam.math.ethz.ch}
\thanks{The second author, the majority of whose work was done whilst at Caltech,
 was also supported by the Department of Defense Vannevar Bush
Faculty Fellowship N00014-22-1-2790 awarded to the first author. The second author is grateful to ETH Z\"urich Department of Mathematics for support during the completion of the manuscript.} 


\date{\today}


\begin{abstract} 
Operator learning in scientific machine learning is concerned
with approximation of maps between infinite-dimensional function spaces; such
 maps frequently arise as the solution operators of partial differential equations (PDEs). Neural operators have demonstrated broad empirical success at approximating such maps from data. However, most existing neural operator architectures enforce boundary conditions indirectly through training from data even though the boundary condition is often known exactly. Furthermore, existing modifications and approaches that do enforce boundary conditions explicitly suffer from impractical restrictions, including boundary smoothness, uniform grids, and separable, box-like domains. In this work, we propose an architecture which, 
independently of training, satisfies homogeneous Dirichlet boundary conditions,
whilst simultaneously retaining the expressivity of existing 
kernel-integral neural operator architectures.
This is achieved by enforcing the property that the output of each
layer is contained in the span of a subset of the homogeneous Dirichlet eigenfunctions 
of the Laplacian on the output domain. The method requires only that the 
output domain be bounded with Lipschitz boundary and places no restriction on the choice of discretization, making it applicable to arbitrary mesh data and general geometries. We prove universal approximation for the resulting architecture;
furthermore the approach we adopt in the analysis proves universality for a broad class of kernel-integral neural operators thereby uniting existing theory for a variety of operator learning methods.  We validate the proposed method on maps defined by the coefficient to solution map in 2D PDEs: Darcy flow on a 
square domain and the Helmholtz equation on a circular domain. Comparisons are
made with alternative methods.
\end{abstract}

\maketitle

\section{Introduction}

\subsection{Motivation and Literature Review} Interest in, and the use of,
neural network-based models for problems in science and engineering has 
exploded in the past decade. For many of these problems, the underlying phenomena are governed by partial differential equations (PDEs), which map input functions to output functions via an infinite-dimensional operator. \textit{Operator learning} encompasses a branch of machine-learning models that generalize to this infinite-dimensional function space setting; the models themselves are termed \textit{neural operators} and early examples include DeepONet \cite{lu2021learning}, PCA-Net \cite{hesthaven_nonintrusive_2018,bhattacharya_model_2021}, random features \cite{nelsen2021random}  and the Fourier Neural Operator (FNO) \cite{li_fourier_2021}.
For overviews of the field see \cite{kovachki2023neural,azizzadenesheli2024neural,kovachki2024operator}.

Operator learning has proved empirically useful in two main cases. First, when the underlying equations of the system are known, operator learning can accelerate scientific computing by serving as a surrogate model for computationally expensive numerical simulations \cite{peng2022rapid, bhattacharya2025learning,azizzadenesheli2024neural}. The second case is where a model is learned from data 
alone, as the governing equations are unknown \cite{yin2022simulating, you2022physics, dong2025data}. However, in many settings, the reality falls in between these two cases; some information is known about the system, such as a boundary condition or conservation law \cite{hansen2023learning}, but the full descriptive equations remain unknown. In these cases, it is natural to want to enforce the known information without impeding the 
expressivity of the architecture or the ability to learn an approximator through stochastic optimization. It is this setting which largely motivates the present work: we \textit{enforce boundary conditions }for neural operators that hold \textit{even without training} while retaining the expressivity of the architecture. Furthermore, we develop the method in a manner that can act on \textit{arbitrary mesh data} and thus on \textit{general geometries} in any dimension. 

Commonly used neural operator architectures rely on data to learn the boundary condition and thus enforce 
it only approximately. For instance, the DeepONet architecture multiplies two neural networks, one of which, the ``branch'' network, constructs basis functions from the input function, and the second, the ``trunk'' network, serves as a means to evaluate the output function on any point in the domain. Thus, the ``trunk'' network is largely responsible for learning the boundary condition; any point on the boundary should evaluate to match the boundary condition. However, there is no current mechanism to enforce this explicitly without fundamentally altering the expressivity of the architecture. Similarly, the FNO processes the input function through finite-dimensional neural networks acting pointwise both on the function and on a learned reweighting of the function in Fourier space via a \textit{kernel-integral operator} \cite{li_fourier_2021}. By design, this architecture is natural for periodic boundary conditions on box-like domains with uniform grids. In practice, a \textit{positional encoding}, or simply the coordinates $(x_1, \dots, x_d)$ of the input domain, are appended to the input function values at that point. Thus, enforcement of non-periodic boundary conditions comes only by learning a map from this input positional encoding to the boundary value from data, and once again there is currently no mechanism to enforce this explicitly. These two architectures are representative of the larger class of neural operator architectures which enforce boundary conditions in a similar manner: PCA-Net, various modifications of DeepONet \cite{abueidda2025deepokan, goswami2023physics}, and other architectures involving kernel-integral operators, including wavelet neural operators \cite{tripura2023} and the Laplace neural operator \cite{cao2024laplace} all require data to enforce boundary conditions through approximation.

Several lines of research have made efforts to enforce boundary conditions explicitly in both operator learning and scientific machine learning (SciML) more broadly. Most notably, work on orthogonal polynomial neural operators (OPNOs) \cite{liu2024render} uses a Shen transform \cite{shen2011spectral} of Chebyshev polynomials to replace the Fourier basis and transform in the FNO with a basis that naturally satisfies Dirichlet, von Neumann, or Robin boundary conditions. The work also develops an efficient algorithm for computing these transforms and demonstrates that the resulting architecture satisfies boundary conditions to machine precision. A drawback of this work is that it requires data to be evaluated on Chebyshev-Gauss-Lobatto points, a variant of Gaussian grids \cite{boyd2001chebyshev}. This issue is remedied by the follow-up work \cite{liu2023spfno} that develops an algorithm to extend the previous method to uniform grids. Both works prove universal approximation properties of their architecture and demonstrate machine precision satisfaction of the boundary conditions in various PDE examples. However, limitations of these methods are that they may only be generalized to separable domains; i.e.~domains that can naturally be viewed as a $d$-dimensional box with uniform or Gaussian grids as opposed to arbitrary mesh coordinate points. A similar approach also appears in recently proposed work that is capable of enforcing the boundary conditions on arbitrarily curved quadrilateral domains only \cite{dong2026novel}. These caveats prevent the methods from extending to more complex geometries, which may well be the setting when direct enforcement of the boundary condition is most valuable. 

Other approaches to enforce boundary conditions in operator learning include BOON \cite{saad2023guiding}, which uses a physics-informed refining procedure to enforce the constraint, as well as a method based on Q-learning that proposes an algorithm that refines the boundary conditions iteratively during training \cite{cohen2023neural}. In non-operator learning approaches to SciML, boundary conditions have also been encoded explicitly in graph neural networks \cite{horie2022physics} and physics-informed neural networks \cite{berrone2023enforcing}. Furthermore, the boundary condition problem has been considered in work that predates most current approaches in SciML, including for data-driven models based on kernel learning \cite{gnecco2013learning} and standard feedforward neural networks \cite{mcfall2009artificial}. The approach in this latter work involves multiplication by a known function that satisfies a zero boundary condition and has also been carried forward to more recent work on convergence of Deep Galerkin Methods (DGMs) to approximate PDEs in SciML \cite{jiang2023global, cohen2026deep}. Another very recent method proposes to learn function extensions separately from the main neural operator model, which allows additional domain generality \cite{mousavi2026imposing}.

\subsection{Contributions}

We make the following contributions to the state-of-the-art, in relation
to enforcing Dirichlet boundary conditions within operator learning,
as overviewed in the preceding subsection:

\begin{enumerate}[label=(C\arabic*)]
    \item We propose an operator learning methodology that automatically satisfies Dirichlet boundary conditions,
    independent of training, by using the Dirichlet eigenfunctions of the Laplacian as a central part of the architecture.
    \item We prove universal approximation for a general class of nonlocal neural operators, subsuming the proposed Dirichlet eigenfunction based architecture
as a special case; these approximation results may be of independent interest 
beyond the application to enforcing Dirichlet boundary conditions. This general class includes neural operators using only a single hidden layer and single basis function as well as neural operators that perform approximation in a particular basis of interest. Sufficient conditions on such basis functions are provided.
    \item We implement the proposed method to approximate two representative
operators defined by coefficient-to-solution maps arising in PDE problems: the Darcy flow equation on a square domain and the Helmholtz equation on a circular domain. We demonstrate the advantage of the proposed method over FNO, as well as over three possible alternative operator learning approaches that also 
automatically satisfy the boundary condition.
\end{enumerate}

Specifically, we introduce a new, kernel-integral based, 
operator learning architecture in which the output space of each layer builds an approximation in 
the basis of Dirichlet eigenfunctions of the Laplacian in the output domain,
contribution (C1). These eigenfunctions satisfy a zero boundary condition and thus, coupled with a zero-preserving activation function and an appropriate fixed bias, create an architecture that automatically satisfies the Dirichlet boundary condition of a PDE. In fact, replacing the final hidden layer of any kernel-integral neural operator, including FNO, with this Laplacian eigenfunction layer achieves the desired boundary condition without training. In addition, as a departure from previous work that automatically satisfies the boundary condition, this approach is applicable to complex geometries and domains; we assume only that the domain is bounded, with Lipschitz boundary. Furthermore, the method requires no restriction on choice of discretization or grid; it may be used on functions defined on arbitrary meshes. Beyond the choice of the eigenfunctions of the Laplacian as the basis elements, we establish universal approximation for a much broader class of potential basis functions, contribution (C2). The approach we adopt suggests practical extensions of the method to other classical basis choices, including various finite element bases. Numerical results demonstrate the practicality
and competitiveness of the proposed approach, contribution (C3).

The remainder of the paper is organized as follows. In Section \ref{sec:prelim}, we define the nonlocal neural operator, as well as the Dirichlet-boundary condition enforcing variant, and detail the assumptions on our choice of basis. In Section \ref{sec:theory}, we present the main theorems as well as critical lemmas supporting the conclusions; we defer the detailed proofs of the lemmas to the appendix. In Section \ref{sec:num-exp}, we showcase the results of various numerical experiments supporting the practical application of the proposed Laplacian eigenfunction layer. We make concluding remarks
in Section \ref{sec:conclusion}. The more technical proofs, and some useful supplementary lemmas, are contained
in the appendix. Appendix \ref{sec:finite-NN} contains lemmas describing finite-dimensional feedforward neural networks; Appendix \ref{appd:prop-proof} contains proofs that various familiar application settings satisfy the assumptions of our main theorems; Appendices \ref{appd:lem-decomp} and \ref{appd:alpha-approx} contain the proofs of the two main lemmas underlying our main results, Theorem \ref{thm:main-single} and Theorem \ref{thm:main-multiple}; and Appendix \ref{sec:eigfunc-Lap} contains key results on the behavior of the eigenfunctions of the Laplacian. 

\section{Definition of Nonlocal Neural Operator}\label{sec:prelim}

In this section we define the \textit{nonlocal neural operator (NNO)} as in \cite{lanthaler2025nonlocality}, repeated here in the context and notation of our work, as well as specify the \textit{Dirichlet layer} that allows for exact enforcement of Dirichlet boundary conditions. We note that the NNO may also be referred to as a \textit{kernel-integral neural operator} in the literature, and that a variety of common operator learning architectures fall into this general setting, including the Fourier Neural Operator (FNO) \cite{li_fourier_2021}. The NNO definition is given in Subsection \ref{ssec:NNO}. In Subsection \ref{ssec:Dir-layer}, we define the Dirichlet layer variant, which is obtained through the particular choice of the eigenfunctions of the Laplacian as the basis functions in the hidden layer of the more general NNO definition.

Throughout this work, $\|\cdot\|$ and $\la\cdot, \cdot \ra$ denote the $L^2(\Omega)$ norm and inner product, and $|\cdot |$ denotes the Euclidean norm. We let $\Omega$ be a bounded, open, Lipschitz domain in $\R^d$ unless otherwise specified. Furthermore, we define $[N] := \{1, \dots, N\}$. We adopt the convention that $C^{s}(\Omega)$ is the Banach space with norm \[\|\eta\|_{C^{s}(\Omega)} = \sum_{|\alpha| \leq s} \sup_{x \in \Omega} |(\partial^\alpha \eta)(x)| < \infty.\]

\subsection{Nonlocal Neural Operator} \label{ssec:NNO}
Let $\cX(\Omega;\R^o)$, $\cY(\Omega;\R^o)$, and $\cV(\Omega;\R^o)$ be Banach spaces of $\R^o$-valued functions over domain $\Omega$. Given a channel width $d_c$, define the following \textit{lifting} and \textit{projection} layers as learnable feedforward neural networks\footnote{Appendix \ref{sec:finite-NN} contains a succinct summary of the feedforward neural network properties needed in our work.} between finite-dimensional Euclidean spaces defined as Nemytskii operators via
\begin{equation} \label{eqn:lifting}
\cR: \cX(\Omega; \R^k) \to \cV(\Omega;\R^{d_c}), \quad u(x) \mapsto R(u(x),x)
\end{equation}
and 
\begin{equation} \label{eqn:projection}
\cQ: \cV(\Omega;\R^{d_c}) \to \cY(\Omega;\R^{k'}), \quad v(x)\mapsto Q(v(x)). 
\end{equation}
Note that $\cR$ includes the positional encoding via pairing of $u(x)$ with the domain variable $x$.
 Each hidden layer $\cL_{\ell}$, $\ell \in [L]$, defines a mapping $\cV(\Omega; \R^{d_c}) \to \cV(\Omega; \R^{d_c})$ via
\begin{equation}
    \label{eqn: hidden-layer}
    (\cL_\ell v)(x) = \sigma\left(W_\ell v(x) + b_\ell + \sum_{m=1}^M\la T_{\ell,m} v, \psi_{\ell,m}\ra_{L^2(\Omega;\R^{d_c})} \phi_{\ell,m}(x)\right). 
\end{equation}
Here the activation function $\sigma: \R \to \R$ is a non-polynomial $C^\infty$ function which is extended to a Nemytskii operator acting component-wise on vector inputs: for a function $v(x) = (v_1(x), \dots, v_{d_c}(x))$, we define $\sigma(v(x)) = (\sigma(v_1(x)), \dots, \sigma(v_{d_c}(x)))$ for $x \in \Omega$. We may now define an NNO; note that the definition includes
an assumption on the properties of the activation function and that this assumption follows the definition.

\begin{definition}[Nonlocal Neural Operator]\label{def:NNO}
The \emph{Nonlocal Neural Operator} (NNO) is a mapping $\Psi: \cX(\Omega;\R^k) \to \cY(\Omega;\R^{k'})$ consisting of a lifting layer $\cR$, hidden layers $\cL_1, \dots, \cL_L$, and projection layer $\cQ$ as defined in equations \eqref{eqn:lifting}, \eqref{eqn: hidden-layer}, and \eqref{eqn:projection}, respectively, that acts on input $a \in \cX(\Omega; \R^k)$ via 
    \begin{equation*}
        \Psi(a) = \cQ \circ \cL_{L} \circ \dots \circ \cL_{1} \circ \cR(a),
    \end{equation*}
    and uses a non-polynomial $C^\infty$ activation function $\sigma$. \hfill$\lozenge$
\end{definition}

  We assume that the activation function $\sigma$ used throughout the model is such that a feedforward, zero-preserving neural network can universally approximate the identity in $C^s$ on a closed ball in a finite dimension. This is true for ReLU for any $s \geq 0$ \cite[Appendix B, Part (iv)]{nakada2020adaptive}, via $\sigma(x) - \sigma(-x) = x$, which allows for exact replication of the identity. It is also true for GeLU for any integer $s \geq 0$ \cite[Lemma 3.1]{yakovlev2025approximation}, which, unlike ReLU, also satisfies the $C^\infty$ condition.

\begin{assumption}\label{ass:sig-eta}
   Let $I_d$ be the identity matrix in $d$ dimensions. Let $s_\sigma \geq 0$. We assume $\sigma$ is such that $\sigma(0)=0$ and for any $\epsilon > 0$ and closed ball $\cK \subset \Rd$, there exists a number of layers $L$, matrices $\{A_l\}_{l=1}^L$, and biases $\{b_l\}_{l=1}^L$ such that for neural network \[G(x) = A_L(\sigma(A_{L-1}\sigma(\dots A_1x + b_1) \dots + b_{L-1}) + b_L,\]
    $G(0) = 0$, and
    it holds that $\|G - I_d\|_{C^{s_\sigma}(\cK)} < \epsilon.$ \hfill$\lozenge$ 
\end{assumption}

Universal approximation results have been proved for many specific variants of the NNO that make a particular choice of the basis functions $\psi_{\ell, m}$ and $\phi_{\ell, m}$ in \eqref{eqn: hidden-layer}. For example, in the FNO these are the Fourier basis elements. In order to capture a broad class of possible architectures under the same theory, rather than proving multiple specific theorems for each specific basis choice, in this work we prove universality for any NNO whose basis functions satisfy mild conditions. Informally, if the map to be approximated produces an output set that can be approximated in some $L^2(\Omega)$ basis $\{\eta_j\}_{j=1}^\infty \subset C^{s'}(\Obar)$, then one can achieve $C^{s'}$ approximation using only a single hidden layer and single pair of basis functions $\psi_{1,1}$, positive almost-everywhere and continuous, and $\phi_{1,1}$, nonzero and $C^{s'}$. This is formalized in Section \ref{sec:theory}. Additionally, if the set of functions $\{\eta_j\}_{j=1}^\infty$ only satisfies $\{\eta_j\}_{j=1}^\infty \subset C^{s'}(\Omega)$, then one can simply choose a finite subset of these functions to serve as the basis functions $\phi_{L,m}$ in the final hidden layer and achieve universal approximation as well. This latter approach can be practically useful when one seeks an output function satisfying certain conditions. These mild conditions also allow one to interchange various choices of basis functions at each layer without violating universality. We emphasize that the general conditions we specify on the basis functions in Section \ref{sec:theory} unite approximation theory for architectures with various choices of basis functions in one single result.

The prior work \cite{lanthaler2025nonlocality} showed that in fact one can set $M=0$, $L=0$, and $\psi_{0,0} = \phi_{0, 0} \equiv 1$ in the NNO architecture and achieve universality via the resulting \textit{averaging neural operator} (ANO). It is a remarkable fact, established in \cite{lanthaler2025nonlocality}, that
a single hidden layer taking the form 
\[
(\cL_1 v)(x) = \sigma\left(Wv(x) + \frac{1}{|\Omega|} \int v(x) \dx \right)
\] 
is sufficient for universal approximation. 
While a surprising and insightful theoretical result, the ANO is not competitive with other operator learning
architectures. Intuitively, the ANO chooses a constant basis to carry the nonlocal information. In the case of, for instance, zero Dirichlet boundary conditions, this adds a constant value everywhere in the domain $\Omega$ and relies heavily on the dependence of the projection layer $Q$ on the positional encoding $x$, carried to the input of $Q$ through an identity block submatrix of $W$. The boundary conditions are enforced through the ability of the ReLU neural network $Q$ to approximate arbitrary continuous functions: positions $x$ on the boundary are mapped close to $0$ as learned from data. In practice, one would never want to use such a restrictive architecture, especially in the case of Dirichlet boundary conditions. In this work we seek a more natural basis for maps with such boundary conditions.

\subsection{Dirichlet Layer}\label{ssec:Dir-layer}
In this subsection, we propose a layer of the nonlocal neural operator that automatically enforces Dirichlet boundary conditions \textit{even before training}. In particular, we set the basis functions equal to the Dirichlet eigenfunctions of the Laplacian on $\Omega$; doing so automatically enforces a zero boundary condition without training, as we prove in Section \ref{sec:theory}. We refer to this modified layer as a \textit{Dirichlet layer}, and define it explicitly below. 
\begin{definition}[Dirichlet Layer]\label{def:Dir-layer}
The \emph{Dirichlet layer} is a mapping $\cV(\Omega; \R^{d_c}) \to \cV(\Omega; \R^{d_c})$ via equation \eqref{eqn: hidden-layer}, where $\psi_{\ell,m} = \phi_{\ell,m} = \eta_{m}$, where $\ell$ is the layer index and $\{\eta_m\}_{m=1}^\infty$ are the Dirichlet eigenfunctions of the Laplacian on $\Omega$. \hfill$\lozenge$
\end{definition}

By using a Dirichlet layer with zero bias as the final hidden layer in an NNO, in tandem with a zero-preserving projection layer $\cQ$, one may enforce homogeneous Dirichlet boundary conditions without training, as we prove in Section \ref{sec:theory}. One may use other choices of hidden layers prior to the Dirichlet layer, and the boundary conditions will still hold. We also point out that while this layer is designed to produce outputs with zero boundary conditions, for constant, nonzero boundary conditions, we may simply preprocess and postprocess the data. For instance, for a map $(u,\mathfrak{b}) \mapsto v$, where $u$ and $v$ are input and output functions, respectively, and $\mathfrak{b}$ is the boundary condition, we may learn the map $(u,\mathfrak{b}) \mapsto v - \mathfrak{b}\mathbf{1}$, where $\mathbf{1}$ is the constant function equal to $1$ on $\Obar$, and then postprocess the output of the NNO by adding back $\mathfrak{b}\mathbf{1}$.

\begin{remark}[Computational Complexity]
     For the FNO, the fast Fourier transform allows the computational complexity of a hidden Fourier layer to be $\mathcal{O}(N\log N)$, where $N$ is the number of grid points, or $N = n^d$ where $n$ is the number of points in a single dimension of the uniform grid. This complexity remains the dominant term even when the number of Fourier modes $K$ is decreased. For a Dirichlet layer with $M$ eigenfunctions forming the hidden layer \eqref{eqn: hidden-layer}, the complexity is $\mathcal{O}(N M)$. In practice, for reasonable mesh sizes, both for FNO and for the Dirichlet layer, we observe that $K \approx \log N$ and $M \approx \log N$ are more than sufficient to achieve low error, and so the Dirichlet layer effectively shares the computational complexity of an FNO layer. This empirical observation is also supported by the numerical experiments in the existing work of \cite{lanthaler2025nonlocality}, which demonstrates that fixing $M$ for the nonlocal neural operator can be optimal in terms of the cost-accuracy tradeoff between parameter count and model accuracy. These findings are corroborated by our own numerical experiments in Section \ref{sec:num-exp}. \hfill$\lozenge$
\end{remark}

\section{Theory}\label{sec:theory}

In this section, we unite theory for a broad class of kernel-integral neural operators under one umbrella. We do so via two theorems. The first theorem, Theorem \ref{thm:main-single}, states that the nonlocal neural operator (NNO) architecture of Definition \ref{def:NNO} is universal with only a \textit{single} hidden layer with a single basis function. This is a generalization of the averaging neural operator result of \cite{lanthaler2025nonlocality}; in particular the single basis function no longer needs to be the constant function.
In fact, universality of several well-known neural operator architectures may be viewed as special cases of this result; see Proposition \ref{prop:cases} for further explanation.

The second main theorem we prove also concerns universality of the NNO architecture, but in a slightly more general setting where we are forced to use the basis functions present in the architecture itself to achieve the universality result. Namely, the first theorem requires that the output space be representable in a basis whose elements are continuous \textit{up to and including} the boundary of the domain $\Omega$. When this is true, finite-dimensional neural networks can approximate, to arbitrary accuracy, continuous functions on bounded sets in $\Rd$, so the proof itself simply uses finite-dimensional neural networks to approximate basis functions. However, this is intuitively unappealing. If the idea is to approximate a function using hidden layers of the form \eqref{eqn: hidden-layer} with some choice of basis elements, it is natural to suppose that the proof could use the fact that the architecture includes these basis elements already instead of reconstructing them using finite-dimensional neural networks. This achievement is the most appealing feature of Theorem \ref{thm:main-multiple}; the proof uses the architecture's own basis functions to obtain the universality result in addition to holding in a slightly more general setting. It is this latter theorem that applies to the Dirichlet eigenfunctions of the Laplacian as well, which for general Lipschitz domains are $C^\infty(\Omega)$ but not necessarily $C^{s'}(\Obar)$ for any $s' > 0$.

We require the following critical assumption about the ability of the chosen basis to approximate functions in the desired output set. Intuitively, one is allowed to first apply a smoother to the output set to be approximated, but then the chosen basis must uniformly approximate the smoothed set. Since we view the following assumption as rather general, we detail in Proposition \ref{prop:cases} some familiar or relevant cases in which Assumption \ref{ass:smoothed_appr} holds. This assumption may be viewed informally as a consequence of separability of the relevant output space, compactness of the input set, and continuity of the underlying map of interest in relevant application settings. Since it is used in the context of the output space in subsequent results, we keep the ``primed'' notation for $\sfK'$ and $s'$, even though the assumption itself is self-contained. Recall that, unless otherwise specified
in this paper, $\Omega$ is open; however the following assumption  and proposition are formulated 
in a slightly more general setting.

\begin{assumption}[Smoothed Approximation]\label{ass:smoothed_appr}
    Let $\Omega \subset \R^d$ be a bounded set that may be open or closed. A compact subset $\sfK'$ of Banach space $\cB(\Omega)$, continuously embedded in $L^1(\Omega)$, satisfies the smoothed approximation assumption in $\cB(\Omega)$ with respect to a set of functions $\{\eta_j\}_{j=1}^\infty$, if there exists a set $B \subset \Rd$ such that $\{\eta_j\}_{j=1}^\infty$ are defined on $B \supseteq \Omega$, there exists a family of linear maps $\{f_\delta\}_{\delta > 0}$, also defined on domain $B$, such that $f_\delta: \cB(\Omega) \to \cB(\Omega)$ for $\delta >0$, and the following hold:
    \begin{enumerate}[label=(S\arabic*)]
        \item\label{it:smooth-conv} \textit{(smoother convergence)} $\lim_{\delta \to 0} \sup_{w \in \sfK'} \|f_\delta(w) - w\|_{\cB(\Omega)} = 0.$
        \item\label{it:proj-conv} \textit{(projection convergence)} For each $\delta >0$, 
        \[\lim_{J \to \infty} \sup_{w \in \sfK'} \Big\| f_\delta(w) - \sum_{j=1}^J \la f_\delta(w), \eta_j \ra_{L^2(B)} \eta_j\Big\| _{\cB(\Omega)} = 0.\]
        \item \label{it:cont-func}\textit{(continuous functional)} For each $\delta >0$ 
and $j \in \mathbb{N}$, the map $w \mapsto \la f_\delta(w), \eta_j \ra_{L^2(B)}$ is a continuous functional on $\cB(\Omega)$. 
    \end{enumerate} \hfill$\lozenge$
\end{assumption}

To establish some key examples where the Smoothed Approximation Assumption above holds, we present the following proposition. 
\begin{proposition}\label{prop:cases}
    The following pairs of functions and Banach spaces satisfy the Smoothed Approximation Assumption \ref{ass:smoothed_appr}: 
    \begin{enumerate}
        \item \label{it:FNO}Let $\Omega$ be a bounded set in $\Rd$ with Lipschitz boundary. For a box $B$ containing $\overline{\Omega}$ in its interior, let $\{\eta_j\}_{j=1}^\infty$ be the $L^2$-orthogonal Fourier sine/cosine basis in $L^2_{\text{per}}(B)$, and let $\sfK'$ be any compact subset of  $C^{s'}(\Obar)$ for integer $s' \geq 0$. Then $\sfK'$ satisfies the Smoothed Approximation Assumption \ref{ass:smoothed_appr} in $C^{s'}(\Obar)$ with respect to 
        $\{\eta_j\}_{j=1}^\infty$.
        \item \label{it:DNO}Let $\Omega$ be a bounded set in $\Rd$ with Lipschitz boundary. Let $\{\eta_j\}_{j=1}^\infty$ be Dirichlet eigenfunctions of the Laplacian on $B:=\overline{\Omega}$. Let $\sfK'$ be any compact subset of $C_0(\Obar)$. Then $\sfK'$ satisfies the Smoothed Approximation Assumption \ref{ass:smoothed_appr} in $C_0(\Obar)$ with respect to $\{\eta_j\}_{j=1}^\infty$. 
        \item \label{it:Leg} Let $\Omega = [-1,1]^d$. Let $\{\eta_j\}_{j=1}^\infty$ be the tensorized Legendre polynomials defined on $B:=\Omega.$ Let $\sfK'$ be any compact subset of $C^{s'}(\Omega)$ for integer $s' \geq 0$. Then $\sfK'$ satisfies the Smoothed Approximation Assumption \ref{ass:smoothed_appr} in $C^{s'}(\Omega)$ with respect to $\{\eta_j\}_{j=1}^\infty$. 
    \end{enumerate}
\end{proposition}
Note that Case \ref{it:FNO} encompasses the averaging neural operator and the FNO setting for $C^{s'}(\Obar)$ approximation. We prove later that Case \ref{it:DNO} covers the Dirichlet layer setting. Case \ref{it:Leg} is the unweighted $L^2$-norm analogue of the tensorized Chebyshev polynomial basis used in \cite{liu2024render}. Ultimately, all these settings are covered  by the same theory in Theorems \ref{thm:main-single} and \ref{thm:main-multiple}. Since the proof of Proposition \ref{prop:cases} is technical, we defer it to Appendix \ref{appd:prop-proof}. The intuition behind each case is that we first apply a smoother to an extension of the set $\sfK'$ and then we may approximate it in the specified basis to obtain Smoothed Approximation Assumption \ref{ass:smoothed_appr}, \ref{it:proj-conv}. The Smoothed Approximation Assumption \ref{ass:smoothed_appr}, conditions \ref{it:smooth-conv} and \ref{it:cont-func}, are obtained from properties of the smoothing and extension operators used.

\subsection{Main Theorems}

In this subsection we present our main theorems, Theorem \ref{thm:main-single} and \ref{thm:main-multiple}, which prove universality for NNOs under general conditions on the basis functions. Through this result, we unite theory for kernel-integral neural operators under two statements. In Corollary \ref{cor:Laplacian-version}, we observe that enforcing constant Dirichlet boundary conditions through the Dirichlet layer in Definition \ref{def:Dir-layer} retains universality while automatically enforcing the boundary condition.

We will use the following assumptions.

\begin{assumption}[Single Basis Pair]\label{ass:single-basis-pair}
    We assume that $\eta_1$ and $\eta_2$ satisfy 
    \begin{enumerate}[label=(E\arabic*)]
        \item \label{enum:eta1-1} $\eta_1 \geq 0 $ on $\Omega$.
        \item \label{enum:eta1-2}$\eta_1 = 0$ only on a set of Lebesgue measure $0$.
        \item \label{enum:eta1-3} $\sup_{y \in \Omega}|\eta_1(y)| < \infty$.
        \item \label{enum:eta1-4} $\eta_1 \in C(\Omega)$.
        \item \label{enum:eta2-5} $\eta_2 \in C^{s'}(\Obar)$ and $\eta_2 \neq 0$ on $\Obar$. 
    \end{enumerate} \hfill$\lozenge$
\end{assumption}

The first theorem addresses the case when approximation is possible using only a single hidden layer and pair of basis functions.
\begin{theorem}[Single Basis Element]\label{thm:main-single}
    Let $\Omega$ be a bounded domain in $\Rd$ with Lipschitz boundary. Let $\Psi^\dagger: C^s(\Obar; \R^k) \to C^{s'}(\Obar; \R^{k'})$ be a continuous operator for integers $s,s'\geq 0$, and $\sfK \subset C^{s}(\Obar; \R^k)$ be a compact subset. Assume that $\Psi^\dagger(\sfK)$ satisfies the Smoothed Approximation Assumption \ref{ass:smoothed_appr} in $C^{s'}(\Omega; \R^{k'})$ with respect to some sequence of functions $\{\xi_j\}_{j=1}^\infty$. If these functions are also such that $\{\xi_j\}_{j=1}^\infty \subset C^{s'}(\Obar)$, then for any $\varepsilon >0$, there exists a nonlocal neural operator $\Psi$ as defined in Definition \ref{def:NNO}, with Assumption \ref{ass:sig-eta} on the activation $\sigma$ satisfied with $s_\sigma = s'$, such that 
    \begin{equation}\label{eqn:thm-main-single}
        \sup_{u \in \sfK} \|\Psi^\dagger(u) - \Psi(u)\|_{C^{s'}(\Omega)} < \varepsilon,
    \end{equation}
    with a single hidden layer $(M=1)$ with any single basis function pair 
$\psi_{1,1} := \eta_1$ and $\phi_{1,1} := \eta_2$ satisfying 
Assumptions \ref{ass:single-basis-pair}.
\end{theorem}

The second theorem addresses the slightly more general case where the Smoothed Approximation Assumption only holds for functions in $C^{s'}(\Omega)$ rather than in $C^{s'}(\Obar)$. It also builds the approximation in the proof by using basis functions \textit{contained in the NNO architecture itself}. This justifies the practical intuition that by choosing a set of basis functions to use in the NNO architecture, one is projecting onto the span of the functions as part of the approximation.
\begin{theorem}[Approximation Using NNO Basis Elements]\label{thm:main-multiple}
    Let $\Omega$ be a bounded domain in $\Rd$ with Lipschitz boundary. Let $\Psi^\dagger: C^s(\Obar;\R^k) \to C^{s'}(\Obar; \R^{k'})$ be a continuous operator for integers $s,s'\geq 0$, and $\sfK \subset C^{s}(\Obar;\R^k)$ be a compact subset. Assume that $\Psi^\dagger$ satisfies the Smoothed Approximation Assumption \ref{ass:smoothed_appr} in $C^{s'}(\Omega; \R^{k'})$ with respect to some sequence of functions $\{\xi_j\}_{j=1}^\infty$. If these functions are also such that $\{\xi_j\}_{j=1}^\infty \subset C^{s'}(\Omega)$, then for any $\varepsilon >0$, there exists a nonlocal neural operator $\Psi$ as defined in Definition \ref{def:NNO}, with Assumption \ref{ass:sig-eta} on the activation $\sigma$ satisfied with $s_\sigma = s'$, such that 
    \begin{equation}\label{eqn:thm-main-multiple}
        \sup_{u \in \sfK} \|\Psi^\dagger(u) - \Psi(u)\|_{C^{s'}(\Omega)} < \varepsilon,
    \end{equation}
    where the final hidden layer contains $\{\xi_j\}_{j=1}^J$ as outer basis elements for some $J$; i.e.~$\{\xi_j\}_{j=1}^J \subset \{\phi_{L,m}\}_{m=1}^M$ with corresponding inner basis functions $\{\psi_m: \phi_m \equiv \xi_j, j \leq J\}$ satisfying \ref{enum:eta1-2}- \ref{enum:eta1-4} in Assumptions \ref{ass:single-basis-pair} and an earlier hidden layer contains at least one basis function pair $\eta_1$ satisfying \ref{enum:eta1-1}-\ref{enum:eta1-4} and $\eta_2$ satisfying \ref{enum:eta1-2}-\ref{enum:eta1-4}.
     
\end{theorem}

The following corollary justifies the use of the Dirichlet layer for approximation as well as shows that it enforces zero boundary conditions exactly. 

\begin{corollary}\label{cor:Laplacian-version}
    Let $\Omega$ be a bounded Lipschitz domain in $\Rd$, and let $\Psi^\dagger: C^s(\Obar; \R^k) \to C_0(\Obar; \R^{k'})$ be a continuous operator for nonnegative integer $s$. Let $\{\eta_j\}_{j=1}^\infty$ be the eigenfunctions of the Dirichlet Laplacian on $\Omega$. Let $\sfK \subset C^s(\Obar; \R^k)$ be a compact subset. Then for any $\varepsilon >0$, there exists a nonlocal neural operator $\Psi$ as defined in Definition \ref{def:NNO}, with Assumption \ref{ass:sig-eta} on the activation $\sigma$ satisfied with $s_\sigma = 0$, such that
    \begin{equation*}
        \sup_{u \in \sfK}\|\Psi^\dagger(u) - \Psi(u)\|_{C(\Omega)} < \varepsilon,
    \end{equation*}
    with the final NNO layer a Dirichlet layer as defined in Definition \ref{def:Dir-layer}, and 
    \[\sup_{y \in \p \Omega } |(T^* \Psi)(y)| = 0,\]
    where $T^*: H^1(\Omega) \to L^2(\p\Omega)$ is the trace operator. 
\end{corollary}
\begin{proof}
    By Proposition \ref{prop:cases}, Case 2, $\{\eta_j\}_{j=1}^\infty$ satisfy the Smoothed Approximation Assumption with respect to $C_0(\overline{\Omega})$. Furthermore, the eigenfunctions of the Dirichlet Laplacian are in $C^\infty(\Omega) \cap H^1_0(\Omega)$, so Theorem \ref{thm:main-multiple} applies. We note that the ground state eigenfunction $\eta_1$ satisfies conditions \ref{enum:eta1-1}-\ref{enum:eta1-4} \cite[Section 6.5]{evans2022partial}. 
 By the construction in Lemma \ref{lemma:alpha-approx}, the final hidden layer uses $\{\eta_j\}_{j=1}^J$ as basis elements, with the remaining layers a finite-dimensional neural network that preserves $0$ with smooth activation. Thus, $\Psi$ enforces zero boundary condition on the output function in the trace sense. 
\end{proof}

\begin{remark}
The approximation in Corollary \ref{cor:Laplacian-version} is stated for $C(\Omega)$ only. For general Lipschitz $\Omega$, while the eigenfunctions are in $C^\infty(\Omega)$, the norm $\|\eta_j\|_{C^{s'}(\Omega)}$ can blow up as one approaches the boundary, for instance, at reentrant corners. This causes the Dirichlet eigenfunctions of the Laplacian and the space $C^{s'}_0(\Obar)$ to fail the Smoothed Approximation Assumption without additional assumptions on boundary regularity. \hfill$\lozenge$
\end{remark}

The key lemmas underpinning the proofs of Theorems \ref{thm:main-single} and \ref{thm:main-multiple} are Lemma \ref{lem:decomp}, which decomposes the true map into a sum of continuous functionals over the chosen basis, and Lemmas \ref{lemma:alpha-approx-single} and \ref{lemma:alpha-approx}, which are devoted to
approximation of the continuous functionals by an NNO with the architectures corresponding to each theorem. In Subsections \ref{ssec:decomp} and \ref{ssec:NNO-approx}, respectively, we describe and state these lemmas before returning to prove the two theorems in Subsection \ref{ssec:thm-proof}.

\subsection{True Map: Approximation by Sum of Continuous Functions} \label{ssec:decomp}

The first step to establishing the approximation result is Lemma \ref{lem:decomp}, which decomposes the true map $\Psi^\dagger$ into a finite sum of our basis elements $\eta_j$ with coefficients determined by continuous functionals on the input function. This step is similar to that of \cite{lanthaler2025nonlocality}, except that we allow for the choice of basis to be more general than that of the Fourier basis. We note that this result may be useful beyond the setting of this work since one could substitute a variety of bases of interest; some are detailed in Proposition \ref{prop:cases}.

The following lemma decomposes the true map into a sum of the basis elements $\{\eta_j\}_{j=1}^J$ weighted by continuous functionals from $L^1(\Omega;\R^k) \to \R$. 

\begin{restatable}{lemma}{lemdecomp}\label{lem:decomp}
    Let $\Omega$ be a bounded Lipschitz domain in $\Rd$, and let $s, s'$ be nonnegative integers. Let $\Psi^\dagger: C^s(\overline{\Omega}; \R^k) \to C^{s'}(\overline{\Omega};\R^{k'})$ be a continuous operator, and let 
$\{\eta_j\}_{j=1}^\infty \subset C^{s'}(\Omega)$. Let $\sfK \subset C^s(\overline{\Omega};\R^k)$ be a compact subset. Assume that 
$\sfK':=\Psi^\dagger(\sfK)$ satisfies the Smoothed Approximation Assumption \ref{ass:smoothed_appr} 
in $C^{s'}(\Omega)$ 
with respect to $\{\eta_j\}_{j=1}^\infty$. Then for any $\epsilon >0$, there exists $J \in \N$, and continuous functionals $\alpha_1, \dots, \alpha_J: L^1(\Omega; \R^k) \to \R$ such that the operator $\tilde{\Psi}^\dagger: L^1(\Omega;\R^k) \to C^{s'}(\Omega; \R^{k'})$, defined by
    \[\tilde{\Psi}^\dagger(u) = \sum_{j=1}^J \alpha_j(u)\eta_j\]
    satisfies
    \[\sup_{u \in \sfK} \|\Psi^\dagger(u) - \tilde{\Psi}^\dagger(u)\|_{C^{s'}(\Omega)} \leq \epsilon.\]
\end{restatable}

\begin{remark} 
The purpose of defining the approximation $\tilde{\Psi}^\dagger$ as a mapping from $L^1(\Omega;\R^k)$ is to avoid needing to modify the proof for various values of $s$. Note that $C^s(\Omega;\R^k)$ is continuously embedded in $ L^1(\Omega;\R^k)$. \hfill$\lozenge$
\end{remark}

We defer the proof of Lemma \ref{lem:decomp} to Appendix \ref{appd:lem-decomp}.

\subsection{NNO Approximation of Functionals 
and Basis Functions}\label{ssec:NNO-approx}

In this subsection we describe and state the two lemmas that are key to proving Theorems \ref{thm:main-single} and \ref{thm:main-multiple}. These lemmas show that we can approximate each term $\alpha_j(u)\eta_j$ in the sum in Lemma \ref{lem:decomp} by a NNO. The first lemma, Lemma \ref{lemma:alpha-approx-single}, achieves the approximation in the setting of a single hidden layer and basis function pair, corresponding to Theorem \ref{thm:main-single}. The second lemma, Lemma \ref{lemma:alpha-approx}, applies to the slightly more general setting of Theorem \ref{thm:main-multiple} and uses the basis functions in the NNO architecture to represent the $\eta_j$ in \ref{lem:decomp}. The proofs of both lemmas may be found in Appendix \ref{appd:alpha-approx}.

\begin{restatable}{lemma}{lemalphaapproxsingle}\label{lemma:alpha-approx-single}
    Let $\alpha: L^1(\Omega; \R^k) \to \R$ be a continuous functional and let $\sfK \subset L^1(\Omega; \R^k)$ be a compact set consisting of bounded functions $\sup_{u \in \sfK}\|u\|_{L^\infty(\Omega)}< \infty$. Assume $\Omega \subset \Rd$ is a bounded Lipschitz domain. Assume, in addition, that $\eta \in C^{s'}(\Obar; \R^{k'})$ for integer $s' \geq 0$.
Then for any $\varepsilon >0$ there exists a nonlocal neural operator $\tilde{\alpha}: L^1(\Omega; \R^k) \to C^{s'}(\Omega; \R^{k'})$ as defined in Definition \ref{def:NNO}, with Assumption \ref{ass:sig-eta} on the activation satisfied with $s_\sigma = s'$, such that 
    \begin{equation}\label{eqn:lemma-alpha-approx-single}
        \sup_{u \in \sfK} \|\alpha(u) \eta - \tilde{\alpha}(u)\|_{C^{s'}(\Omega)} < \varepsilon
    \end{equation}
with a single hidden layer $(M=1)$ and with any single basis function pair $\psi_{1,1} := \eta_1$ and $\phi_{1,1} := \eta_2$ satisfying Assumptions \ref{ass:single-basis-pair}.
    
\end{restatable}

\begin{restatable}{lemma}{lemalphaapprox}\label{lemma:alpha-approx}
    Let $\alpha: L^1(\Omega; \R^k) \to \R$ be a continuous functional and let $\sfK \subset L^1(\Omega; \R^k)$ be a compact set consisting of bounded functions $\sup_{u \in \sfK}\|u\|_{L^\infty(\Omega)}< \infty$.  Assume $\Omega \subset \Rd$ is a bounded Lipschitz domain. Assume, in addition, that
 $\eta \in C^{s'}(\Omega; \R^{k'})$ for integer $s' \geq 0$. Then for any $\varepsilon >0$ there exists a nonlocal neural operator $\tilde{\alpha}: L^1(\Omega; \R^k) \to C^{s'}(\Omega; \R^{k'})$ as defined in Definition \ref{def:NNO}, with Assumption \ref{ass:sig-eta} on the activation satisfied with $s_\sigma = s'$, such that 
    \begin{equation}\label{eqn:lemma-alpha-approx}
        \sup_{u \in \sfK} \|\alpha(u) \eta - \tilde{\alpha}(u)\|_{C^{s'}(\Omega)} < \varepsilon,
    \end{equation}
where the final hidden layer contains $\eta$ as an outer basis element, i.e.~for some $m^* \leq M$, $\phi_{L,m^*} := \eta$ with a corresponding inner basis function $\psi_{L,m^*}\in L^1(\Omega)$ satisfying \ref{enum:eta1-2}-\ref{enum:eta1-4} in Assumptions \ref{ass:single-basis-pair} and an earlier hidden layer contains at least one basis function pair 
$\eta_1$ satisfying \ref{enum:eta1-1}-\ref{enum:eta1-4} and $\eta_2$ satisfying \ref{enum:eta1-2}-\ref{enum:eta1-4}.
\end{restatable}

\begin{remark}
Lemmas \ref{lemma:alpha-approx-single} and \ref{lemma:alpha-approx} are similar, but notable in two main differences. First, Lemma \ref{lemma:alpha-approx} is slightly more general in that it only requires that $\eta$ be $C^{s'}$ on $\Omega$ rather than on the closure of $\Omega$. In return, we sacrifice using only a single hidden layer and single basis function as in Lemma \ref{lemma:alpha-approx-single}. However, the proof for Lemma \ref{lemma:alpha-approx} is appealing from an intuition standpoint because the basis functions \textit{in the NNO itself} are actually used as the approximating basis, whereas in the proof of Lemma \ref{lemma:alpha-approx-single}, as well as several other operator learning approximation proofs, a finite-dimensional neural network is employed to approximate $\eta$ instead of using a basis function that is already part of the architecture. In practice, one never uses a single basis element, so we find it far more natural to view universal approximation as occurring through the lens of Lemma \ref{lemma:alpha-approx}. \hfill$\lozenge$
\end{remark}

\subsection{Proofs of Theorems \ref{thm:main-single} and \ref{thm:main-multiple}}\label{ssec:thm-proof}
We may now prove our main theorems. As the proofs are identical except for the use of different lemmas, we present them as one. 

\begin{proof}[\textit{Proof of Theorems \ref{thm:main-single} and \ref{thm:main-multiple}} ]
Fix any $\epsilon_0 >0$. By Lemma \ref{lem:decomp},
there is $J>0$ such that we may approximate $\Psi^\dagger$ by 
\[\tilde{\Psi}^\dagger(u) = \sum_{j=1}^J \alpha_j(u) \eta_j,\]
to obtain
\[\sup_{u \in \sfK} \|\Psi^\dagger(u) - \tilde{\Psi}^\dagger(u)\|_{C^{s'}(\Omega)} < \epsilon_0.\]
Then by Lemma \ref{lemma:alpha-approx-single} for Theorem \ref{thm:main-single}, respectively Lemma \ref{lemma:alpha-approx} for Theorem \ref{thm:main-multiple}, each $\alpha_j \eta_j$ may be approximated by an NNO $\tilde{\alpha}_j$ satisfying the appropriate architecture conditions corresponding to each theorem-lemma pair such that for any $\epsilon_1 > 0$,
\[ \sup_{u \in \sfK}\|\alpha_j(u)\eta_j - \tilde{\alpha}_j(u)\|_{C^{s'}(\Omega)} \leq \epsilon_1, \quad \text{ for all } j \in [J]. \]
Define $\Psi := \sum_{j=1}^J\tilde{\alpha}_j$, the parallelization and subsequent sum in the final affine layer of the component NNOs. Then we have that 
\begin{align*}
\sup_{u \in \sfK}& \|\Psi^\dagger(u) - \Psi(u) \|_{C^{s'}(\Omega)} = \sup_{u \in \sfK} \|\Psi^\dagger(u) - \sum_{j=1}^J \tilde{\alpha}_j(u)\|_{C^{s'}(\Omega)}\\
& \leq \sup_{u \in \sfK} \|\Psi^\dagger - \sum_{j=1}^J \alpha_j (u) \eta_j\|_{C^{s'}(\Omega)} + \|\sum_{j=1}^J \alpha_j(u) \eta_j - \sum_{j=1}^J \tilde{\alpha}_j(u) \|_{C^{s'}(\Omega)} \\
& \leq \epsilon_0 + J\max_{j \in [J]}\sup_{u \in \sfK} \|\alpha_j(u)\eta_j - \tilde{\alpha}_j(u) \|_{C^{s'}(\Omega)} \\
& \leq\epsilon_0 + J\epsilon_1.
\end{align*}
Choosing $\epsilon_0 < \frac{\epsilon}{2}$ and $\epsilon_1 < \frac{\epsilon}{2J}$ yields the results in \eqref{eqn:thm-main-single} and \eqref{eqn:thm-main-multiple}.
\end{proof}

\section{Numerical Experiments}\label{sec:num-exp}
In this section we consider two operator learning problems defined
by a coefficient to solution map in a PDE, one based on Darcy flow
on a square domain and the other a variant of the Helmholtz equation on a 
circular domain. We implement a neural operator that enforces Dirichlet 
boundary conditions absolutely and test it on the two example problems.
We demonstrate both the feasibility of our proposed method, and its
competitiveness, in terms of cost-accuracy trade-off, in relation to
other pre-existing methods.\footnote{All code used to obtain results in this section may be found at the Github repository 
\textcolor{blue}{\href{https://github.com/mtrautner/DNO}{https://github.com/mtrautner/DNO}}.}

We first detail the architecture implementation in Subsection \ref{ssec:arch}, and then we describe the data generation process for the Darcy and Helmholtz experiments in Subsections \ref{ssec:darcy} and \ref{ssec:helmholtz}, respectively. We present and discuss results from both examples in Subsection \ref{ssec:results-discussion}. 

\subsection{Architecture}\label{ssec:arch}

While the theory applies more generally, in implementation we make the natural choice of projecting onto the eigenfunctions of the Laplacian with Dirichlet boundary conditions on the domain associated with the application problem. In particular, we use the Dirichlet layer in Definition \ref{def:Dir-layer} as the final hidden layer in our architecture. By setting $W_\ell = 0$ and $b_\ell = 0$, this layer will enforce zero Dirichlet boundary conditions \textit{absolutely} for the output function, as proved in Corollary \ref{cor:Laplacian-version}. Additionally, we enforce zero bias in the final neural network layer $\cQ$. GeLU activations are used throughout the model which, we recall, satisfy Assumptions \ref{ass:sig-eta} for any $s_\sigma \geq 0$, including the requirement that $\sigma(0) = 0$.

We compare our method with a standard implementation of the FNO as well as three possible alternative methods to enforce the boundary conditions exactly via projection. Let $\{\eta_j\}_{j=1}^J$ be the first $J$ eigenfunctions of the Laplacian, and let the projection operator $P_J$ be defined via 
\begin{equation}\label{eqn:proj-op}
    P_J v = \sum_{j=1}^J \la \eta_j, v \ra_{L^2(\Omega;\R^{d_c} )} \eta_j. 
\end{equation}

For the first comparison, we project the output of the trained FNO $\Psi_{\text{FNO}}$ \textit{after training} to obtain the prediction $P_J \Psi_{\text{FNO}}(a)$ for input $a$. For the second comparison, we incorporate this projection as part of the pipeline \textit{during training}. For the third comparison, we bypass neural networks entirely and simply perform regularized least squares (ridge regression) on the coefficients of the input and output functions projected onto the first $J$ eigenfunctions of the Laplacian. For this basic comparison, while the input does not satisfy Dirichlet boundary conditions, the eigenfunctions
are still a basis for $L^2(\Omega)$. All of these approaches enforce the boundary condition absolutely. The last one corresponds to making a linear
approximation to the nonlinear map $\Psi$, whilst the proposed method and the other two comparisons
all result in an approximation of the map $\Psi$ which is nonlinear.

In what follows we aim to approximate $\Psi^\dagger: \cX(\Omega;\R^o) \to \cY(\Omega;\R^o)$ acting on $a$,
coefficient of a PDE, to produce $u$, solution of the PDE. We
assume we are given data $\{a_n, u_n\}_{n=1}^N$ with $a_n \sim \mu$ i.i.d.
and $u_n=\Psi^\dagger(a_n).$ The population objective function that
we aim to approximately minimize, by choice of parameters $\theta$ 
in $\Psi(\cdot;\theta): \cX(\Omega;\R^o) \to \cY(\Omega;\R^o)$, is
$$\mathsf{J}(\theta)=\mathbb{E}^{a \sim \mu} \frac{\|\Psi(a;\theta) - \Psi^\dagger(a)\|_{L^2(\Omega)}}{\|\Psi^\dagger(a)\|_{L^2(\Omega)}}.$$ 
In practice we approximate this empirically using the data to obtain
objective
\begin{equation*}
    \mathsf{J}^N(\theta) = \frac{1}{N} \sum_{n=1}^N \frac{\|\Psi(a_n;\theta) - u_n\|_{L^2(\Omega)}}{\|u_n\|_{L^2(\Omega)}}.
\end{equation*}

\subsection{Darcy Flow}\label{ssec:darcy}
Let $\Omega = [0,1]^2$ be the unit square. We are interested in the map from $a \in L^\infty(\Omega; \R)$ to $u \in H_0^1(\Omega; \R)$ given by 
\begin{align*}
    -\nabla \cdot (a \nabla u) & = 1, \quad x \in \Omega \\
    u & = 0, \quad x \in \partial \Omega. 
\end{align*}

In this example, we use the dataset associated with the original work on FNOs \cite{li_fourier_2021}. The coefficients $a$ are randomly generated from
$\mu := \psi_{\#} \cN(0,C)$ where $C=(-\Delta + 9I)^{-2}$ equipped
with zero Neumann boundary conditions to define the inverse of the
identity shifted Laplacian. The map $\psi$ is a Nemytskii operator defined
by a piecewise constant function taking values $3$ and $12$ on the negative and positive parts of the real line, respectively. The solution was obtained via a second-order finite-difference scheme on a $421 \times 421$ grid, as stated in \cite{li_fourier_2021}. In total the data set has $2048$ samples. We set aside $248$ of the samples for validation used only for hyperparameter tuning, and $200$ of the samples for test data, which were not used to choose hyperparameters. 

For this example, since the data are uniform and on the unit square, we used FNO layers in the DNO until the final layer to demonstrate the advantage of the DNO in modifying only the final layer. One can verify from Theorem \ref{thm:main-multiple} that combining different types of layers in this manner still satisfies the conditions for universality.

\subsection{Helmholtz} \label{ssec:helmholtz}
Let $\Omega $ be the unit disk in $\R^2$. We are interested in the map from $a \in L^\infty(\Omega; \R)$ and $b \in \R$ to the solution $u$ given by 
\begin{align*}
    -\Delta u - \omega^2 a u & = 0, \quad x \in \Omega \\
    u & = b, \quad x \in \partial \Omega, 
\end{align*}
where $\omega = \frac{5\pi}{2}$.
The coefficient $a$ is randomly generated as a sum of Gaussians. The number $n$ of Gaussians is uniformly 
sampled from integers in $[[2,7]]$, the centers of the Gaussians are uniformly sampled in $\Omega$, the amplitudes are sampled uniformly $\sim (2,10)$, and the standard deviation is sampled uniformly $\sim (0.05, 0.1)$. Following summation, the resulting field $\tilde{a}$ is normalized via 
\[
a = \frac{\tilde{a} - a_{\min}}{a_{\max} - a_{\min}}.
\]
This ensures that $a$ is non-negative and takes values in $[0,1]$, 
pointwise. The boundary condition $b$ is uniformly sampled $\sim [0.25, 5]$. The numerical solution is obtained via the finite element method in FEniCSx using CG elements of degree 1 on a triangular mesh of the unit disk \cite{baratta2023dolfinx, alnaes2014unified, balay2019petsc}. Since the operator $-\Delta - \omega^2 a$ is not invertible for all choices of $a$, the solution operator is unbounded unless we restrict our choice of input domain $(a,b)$ further. To do so, after obtaining numerical solution $u$, we reject the sample if $\frac{\max_{x \in \Omega}{|u(x)|}}{b} > 10$. This can be thought of as cutting out holes around the singularities in the domain $(a,b)$; it can also be thought of as conditioning the input measure as previously described.
Doing so is necessary for the true solution map to exist on our choice of input set. To adapt the data for the FNO, we lift $b$ to a function by $ b \mapsto b \mathbf{1}$ and concatenate $a$ and $b\mathbf{1}$ as the input to the FNO. For the DNO, we also lift $b \mapsto b\mathbf{1}$ and transform the output data to be $u - b\mathbf{1}$ so that it has zero boundary condition. Then the map $(a, b\mathbf{1})\mapsto u - b\mathbf{1}$ is learned, and in postprocessing we add $b\mathbf{1}$ back to obtain the true solution. For this example, we produced $11000$ data samples and set aside $500$ for validation data and $500$ for test data. 

Since the Helmholtz data are generated on a triangular mesh, we are able to pass this mesh directly to the DNO by using Dirichlet layers as all hidden layers. On the other hand, to adapt the data for the FNO, we interpolate the data to a uniform grid with approximately the same number of interior data points. We then pass to the FNO a box domain containing the circular domain with a zero mask applied outside the circular domain. For FNO, the error comparison of prediction to truth is performed on the interpolated grid instead of once more interpolating the FNO output back to the FEM mesh. In doing so, we give FNO a slight advantage because the error is not further increased via additional interpolation. 

\begin{figure}
    \centering
    \includegraphics[width=0.78\linewidth]{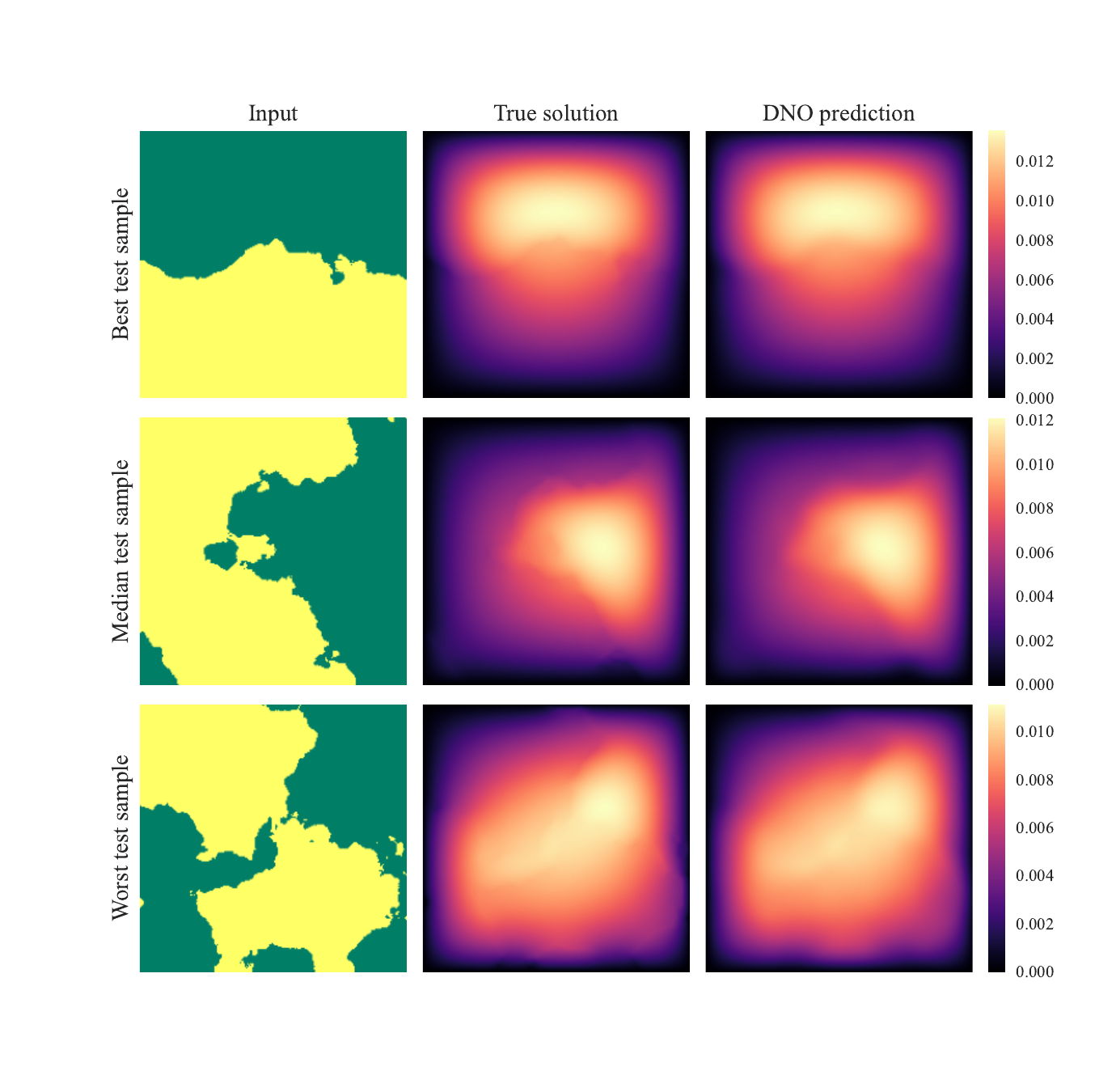}
    \vspace{-1cm}
    \caption{DNO predictions on Darcy flow. Relative $L^2$ error for the best, median, and worst test samples, respectively: $0.482\%$, $0.772\%$, and $2.02\%$.}
    \label{fig:darcy-dno-predictions}
\end{figure}

\begin{figure}
    \centering
    \includegraphics[width=0.68\linewidth]{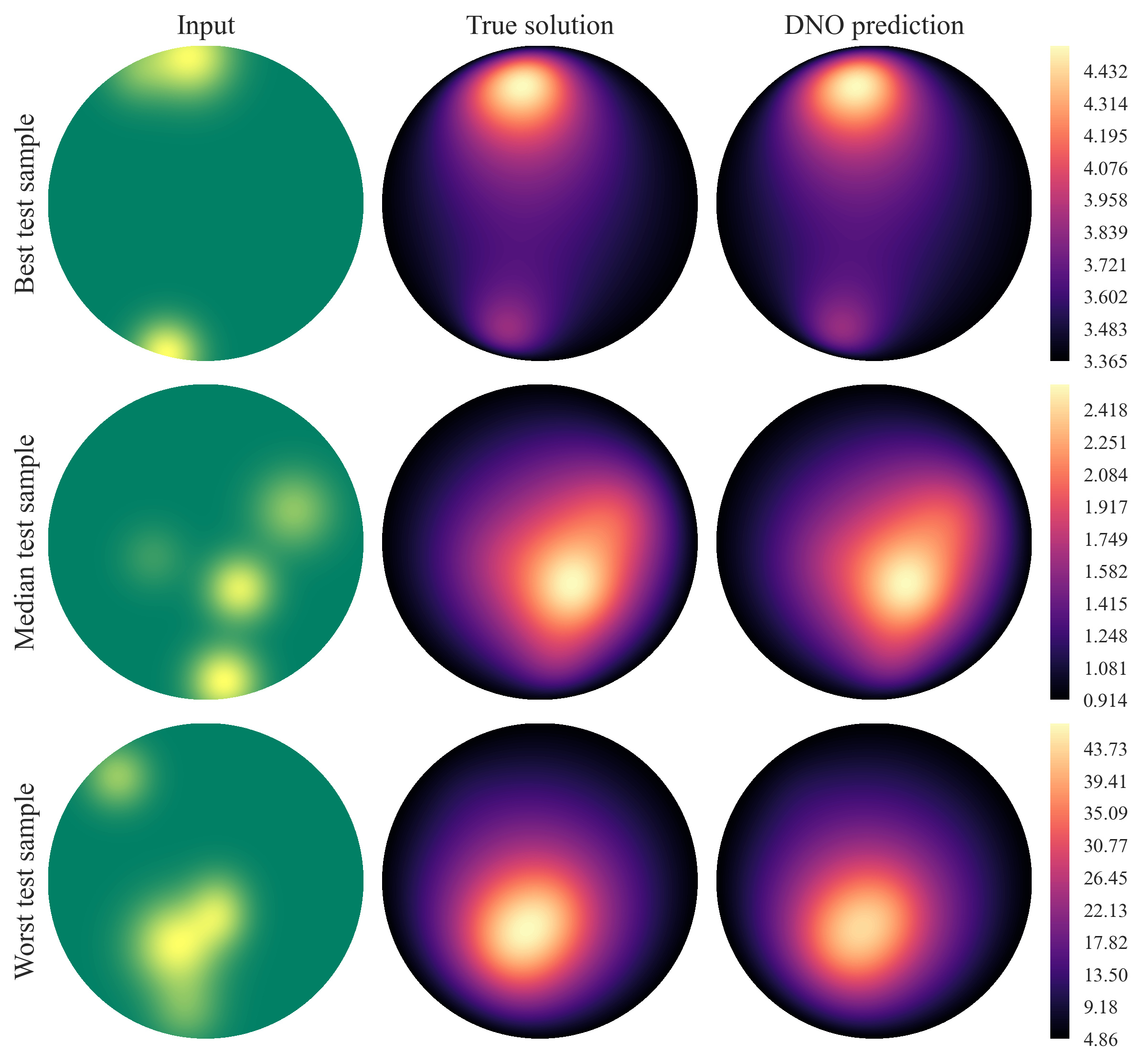}
    \caption{DNO predictions for the Helmholtz example. Relative $L^2$ error for the best, median, and worst test samples, respectively: $0.0283\%$, $0.130\%$, $4.45\%$.}
    \label{fig:helmholtz-dno-predictions}
\end{figure}


\FloatBarrier

\begin{figure}
    \centering
    \includegraphics[width=0.8\linewidth]{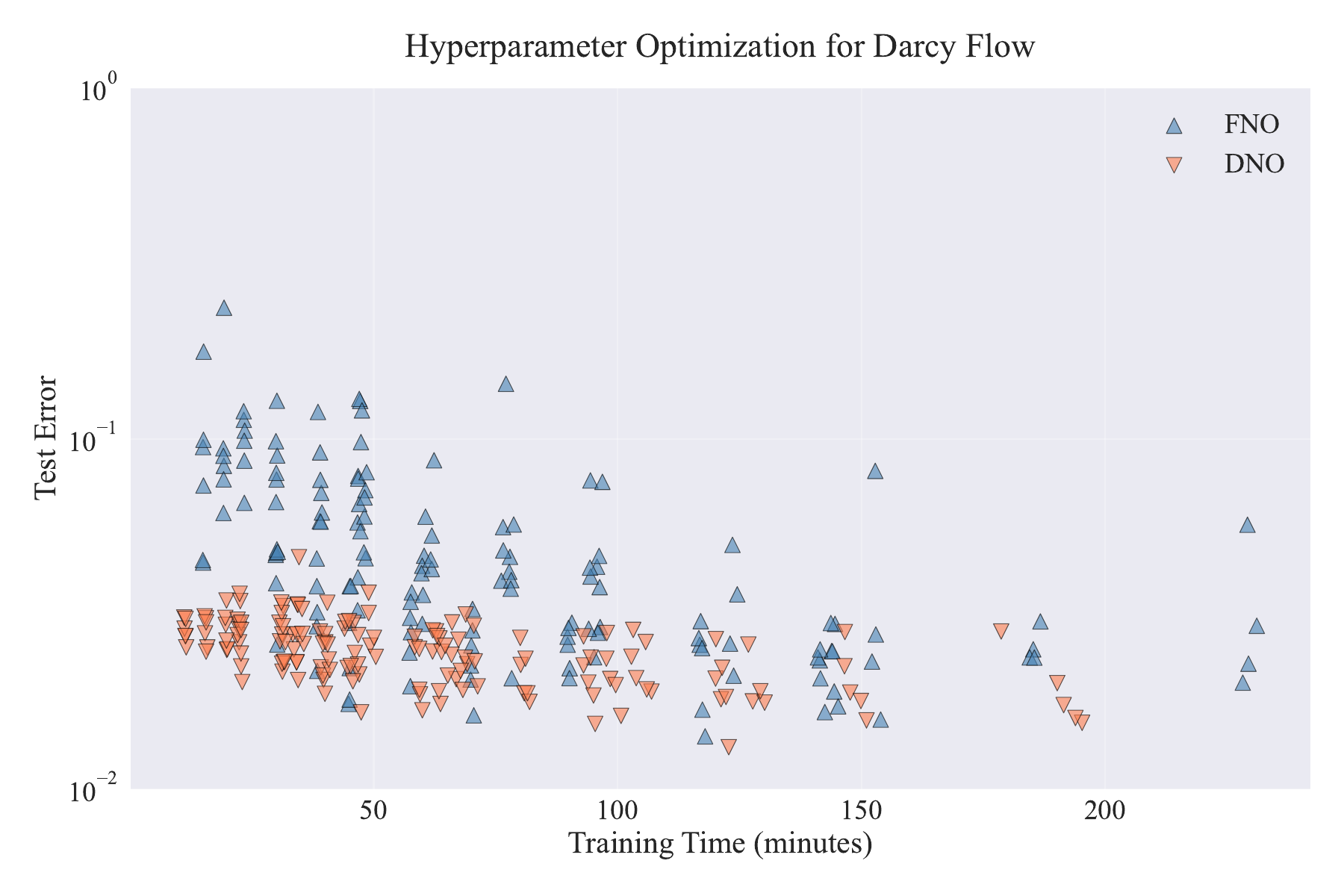}
    \caption{Darcy example hyperparameter search results.}
    \label{fig:darcy_hyperparam}
\end{figure}

\begin{figure}
    \centering
    \includegraphics[width=0.8\linewidth]{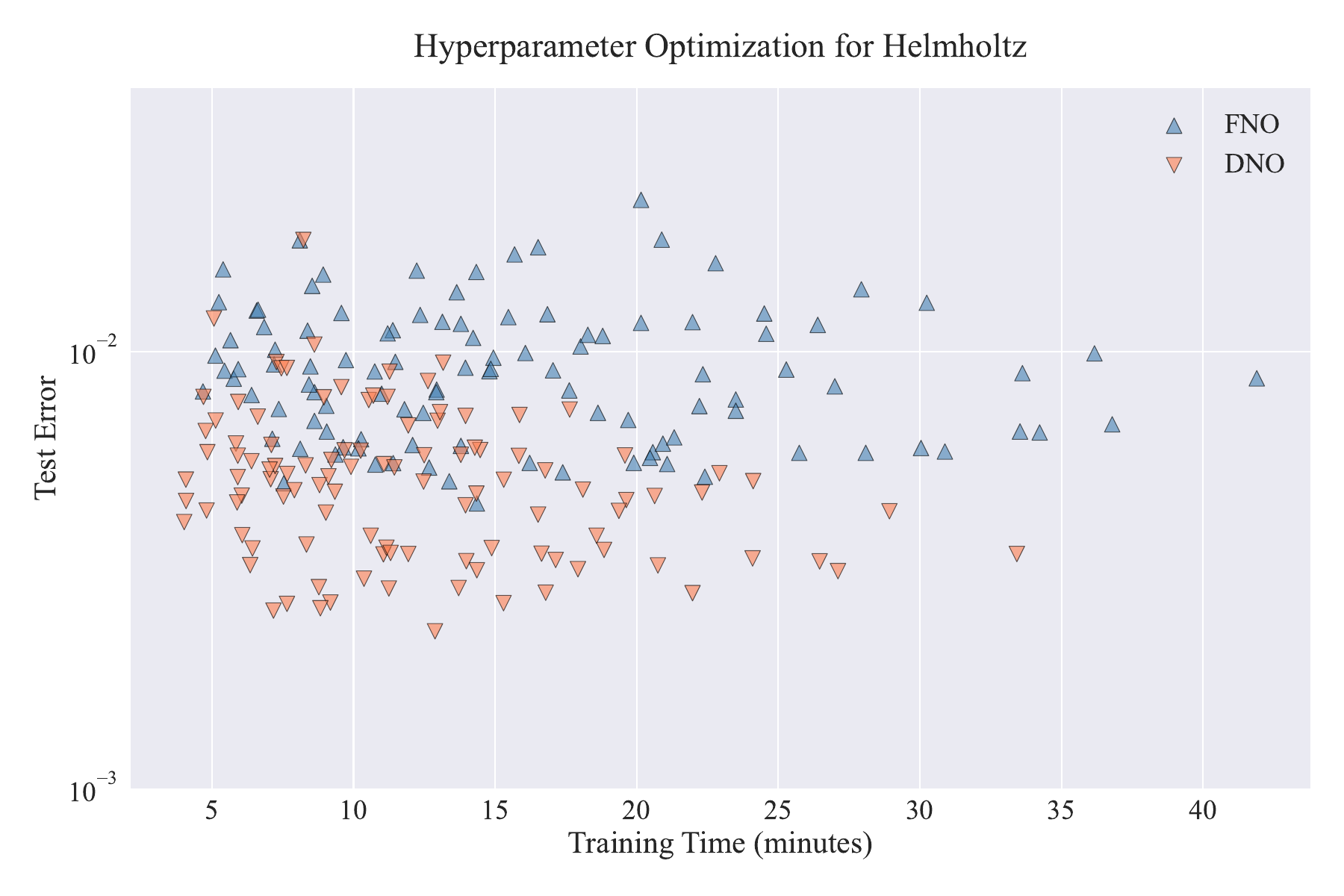}
    \caption{Helmholtz example hyperparameter search results.}
    \label{fig:helmholtz_hyperparam}
\end{figure}

\begin{figure}
    \centering
    \includegraphics[width=0.8\linewidth]{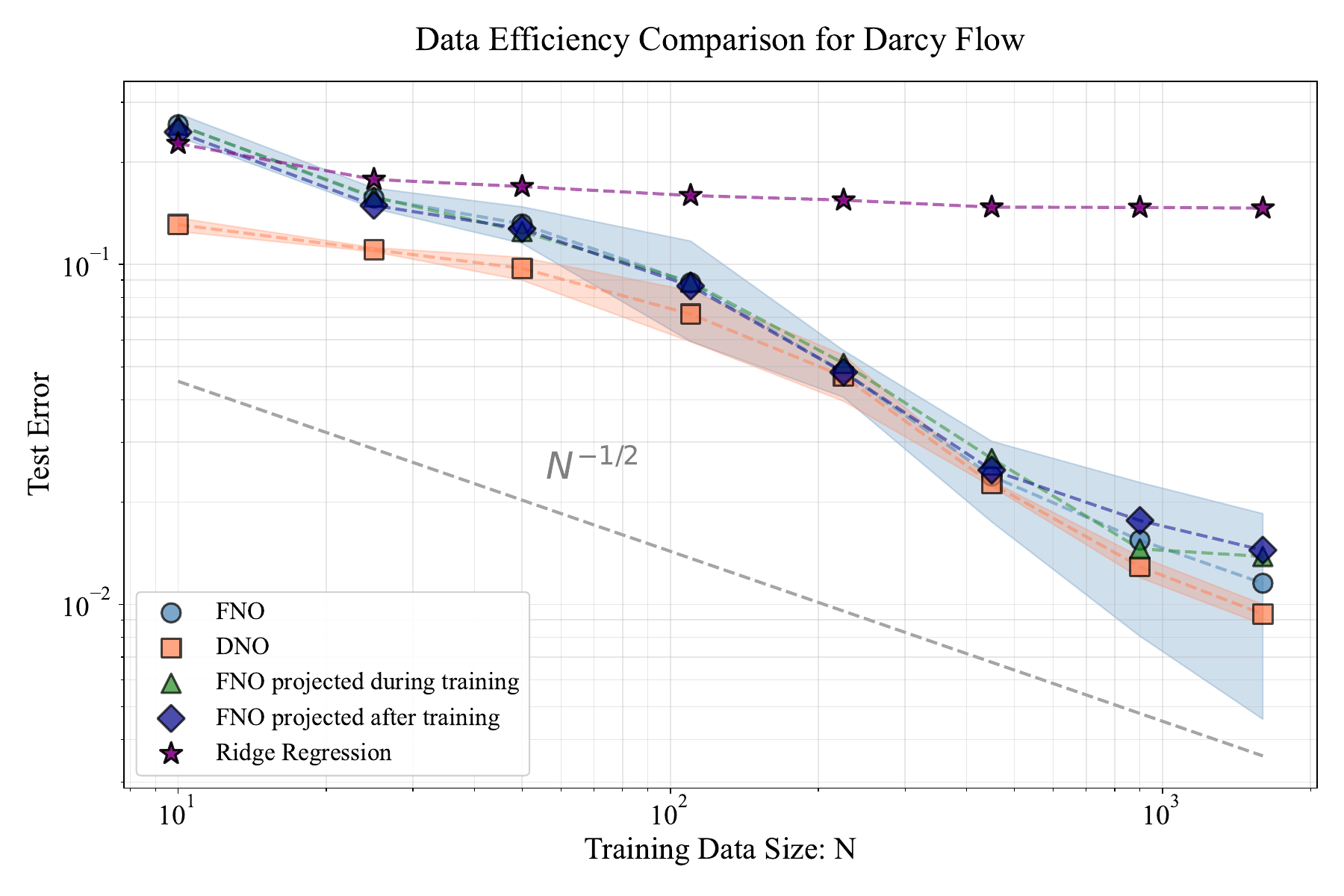}
    \caption{Data efficiency for the Darcy flow example for FNO and DNO as well as projected FNO before and after training.}
    \label{fig:data_efficiency_darcy_extended}
\end{figure}

\begin{figure}
    \centering
    \includegraphics[width=0.8\linewidth]{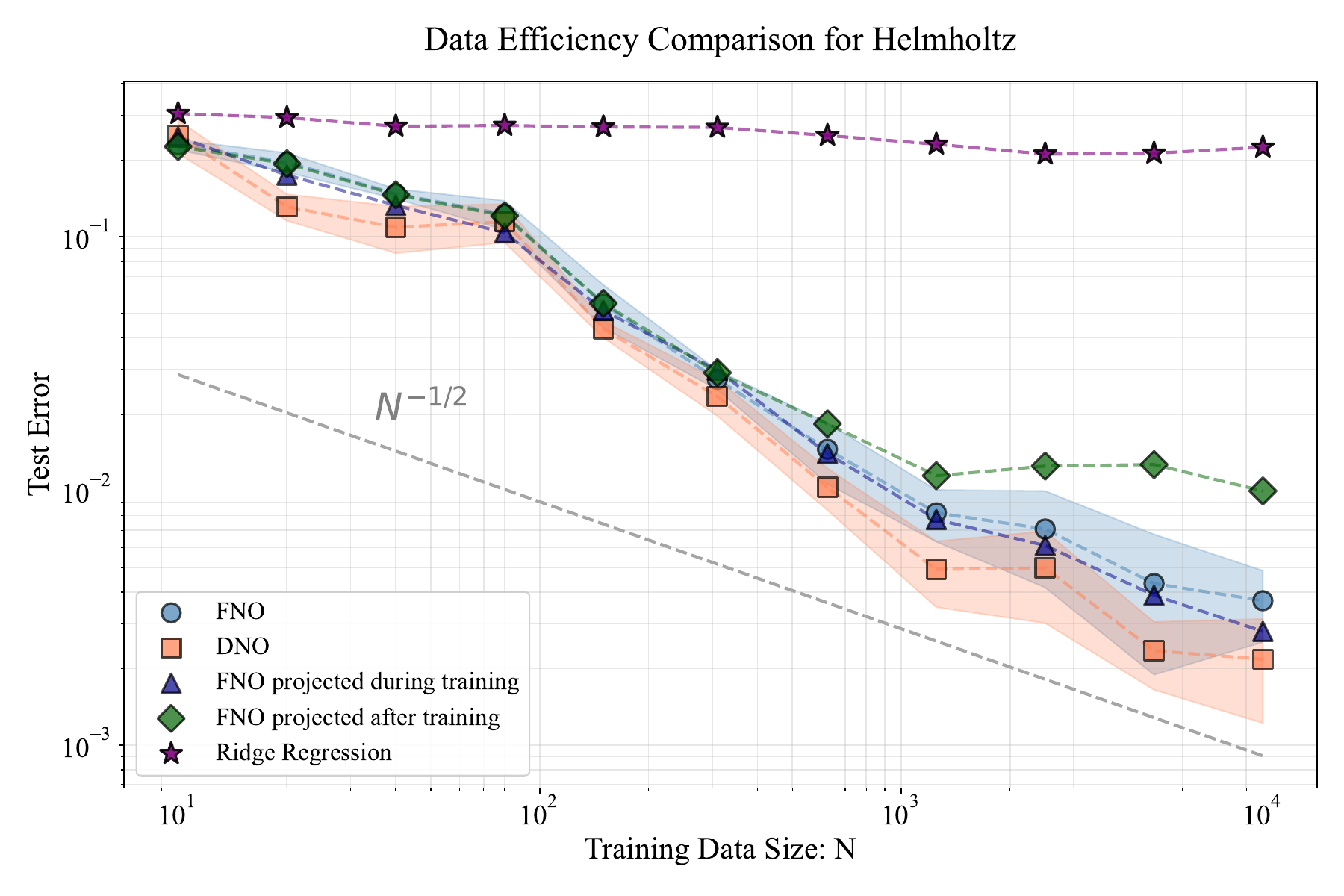}
    \caption{Data efficiency for the Helmholtz example for FNO and DNO as well as projected FNO before and after training and regression.}
    \label{fig:data_efficiency_helmholtz_extended}
\end{figure}

\subsection{Results and Discussion}\label{ssec:results-discussion}

For both PDE examples we ran a hyperparameter gridsearch, doing so
for the DNO and for the FNO. And in both PDE examples, we varied the number of eigenfunctions (resp. Fourier modes) in each dimension, the channel width, the batch size, the number of layers, and the number of training epochs. For the Darcy problem, the resulting choice of hyperparameters for the DNO was $12$ modes in each degree of freedom, a channel width of $64$, $300$ epochs, a batch size of $15$, $4$ FNO layers, and one final DNO layer. For the FNO, the resulting choice was the same with $12$ modes in each dimension, a channel width of $64$, $300$ epochs, a batch size of $14$, and $4$ hidden layers. For the Helmholtz problem, the resulting choices for DNO were $16$ modes, a width of $64$, $300$ epochs, a batch size of $25$, and $4$ DNO layers. For FNO, the choices were $12$ modes, a width of $64$, a batch size of $25$, $300$ epochs, and $4$ layers. We note that the training time of these models were comparable between the FNO and DNO in both examples. A visualization of the input and output functions as well as predictions for both problems can be seen in Figures \ref{fig:darcy-dno-predictions} and \ref{fig:helmholtz-dno-predictions}, for the Darcy and Helmholtz examples, respectively. Figures \ref{fig:darcy_hyperparam} and \ref{fig:helmholtz_hyperparam} give a visualization of the relative validation error versus the training time for the hyperparameter search for the Darcy and Helmholtz examples, respectively. These figures also give empirical evidence that the DNO is more robust to the suboptimal hyperparameter choices that will be made in practice.

We then took models with these hyperparameters and performed data efficiency comparisons of DNO, FNO, and the FNO projected onto the Laplacian eigenfunction basis both before and after training, as well as a basic ridge regression, as detailed in Subsection \ref{ssec:arch}. The results of relative test error versus training data size $N$ are shown in Figures \ref{fig:data_efficiency_darcy_extended} and \ref{fig:data_efficiency_helmholtz_extended}. The shaded uncertainty regions in both figures represent two standard deviations away from the mean of five models trained on the same amount of data; the ridge regression example has no uncertainty as it results from a deterministic linear solve instead of from stochastic gradient descent. The poor performance of ridge regression demonstrates that the map $\Psi$ being learned
is indeed significantly nonlinear. For the Darcy example, the clearest benefit appears in the low data regime, where the DNO shows a 2$\times$ accuracy improvement over any of the other model choices; the other nonlinear model choices fail to beat even ridge regression. The benefit of DNO remains present for higher amounts of data as well; for $1600$ data points the DNO shows a 20\% accuracy improvement over FNO, which in turn beats both projection approaches. Indeed, it is somewhat surprising that neither projection approach improved on the basic FNO implementation; this implies that the benefit of DNO goes beyond simply getting the boundary condition right. For the Helmholtz example, all of the models perform similarly at extremely low levels of data ($N=10$), but the DNO shows roughly a 30\% accuracy improvement over the other models for $N=20$ and $N=40$. Surprisingly, in this example the benefit of DNO shows up most in the higher data regime, showing approximate 2$\times$ improvement in accuracy for $N=5000$ and $N=10000$. In this Helmholtz example, projecting the FNO during training did yield some benefit in the high data regime, but, interestingly, projecting after training caused the error to more than double; we believe this error originates from the interpolation error between the native mesh of the FEM solver and the uniform mesh of the FNO. We emphasize that this phenomenon is not mitigated by increasing $J$ in the projection operator $P_J$. 

Intuitive visualizations of the best, median, and worst test samples are shown for DNO for the Darcy and Helmholtz experiments in Figures \ref{fig:darcy-dno-predictions} and \ref{fig:helmholtz-dno-predictions}, respectively, at training data counts of $N=1600$ and $N=10000$. We comment that for the Helmholtz example, the worst test sample occurs near the cutoff for a singularity as described in Subsection \ref{ssec:helmholtz} as expected; this was also the worst test sample for FNO.

\FloatBarrier
\section{Conclusion} \label{sec:conclusion}

We propose a choice of basis for the nonlocal neural operator that automatically satisfies Dirichlet boundary conditions by using the eigenfunctions of the Laplacian. The approximation theory developed for this approach additionally goes far beyond this specific setting and applies to a broad class of bases for $L^2$ with only Lipschitz boundary assumptions on a domain in any Euclidean space. We implement our proposed ``Dirichlet Neural Operator" on two representative examples, including the Darcy flow problem on a square and the Helmholtz equation on a circular domain. The latter example in particular demonstrates the advantage of the method on arbitrary mesh points as well as with nonzero Dirichlet boundary. The ability of the method to enforce boundary conditions without training on arbitrary meshes and domains distinguishes it from prior work on the subject. 

Several possible extensions are suggested by the work. First, since the theory applies to a broad class of possible bases, additional basis choices such as finite element bases could be implemented to achieve other desirable properties such as different boundary conditions or possibly boundedness guarantees. The choice of Fourier basis for the FNO has allowed for numerical analysis yielding rates in terms of model size parameters; similar rates could be achieved in different settings depending on the choice of basis. Furthermore, it remains an open question how to automatically enforce mixed boundary conditions on arbitrary domains that are defined separately on a partition $\Gamma_1 \cup \dots \cup \Gamma_J = \partial \Omega$ with operator learning.

\section*{Acknowledgments}
The authors thank
Manuel Santana for advice on the numerical solution of the Helmholtz equation and Samuel Lanthaler for technical insights and advice.

AI acknowledgment: Claude Sonnet 4.6 was used to identify sources for theory on the eigenfunctions of the Laplacian. Claude Opus 4.8 made a key suggestion to achieve the proof of Lemma \ref{lemma:alpha-approx-single}; namely, the exact ``$1$'' channel. Cursor was used as a code development aid, including in the numerical solution of the Helmholtz equation, implementation of the Dirichlet layer, and figure generation scripts. Claude Code with Opus 4.8 was used as a code development aid for the linear regression comparison. The authors 
take full responsibility for all the theoretical and numerical results 
in this manuscript. Furthermore, AI was not used to write any portion of the 
manuscript nor to determine which references to cite.

\bibliographystyle{amsplain}
\bibliography{references}

\appendix 
\section{Properties of finite-dimensional neural networks}\label{sec:finite-NN}

\begin{definition}\label{def:finite-NN}
    A \textit{finite-dimensional feedforward neural network} $G: \R^{d_x} \to \R^{d_y}$ takes the form of compositions 
    \[G(x) = A_{L} L_{L-1} \circ \dots \circ L_1(x) + b_{L}, \quad x \in \R^{d_x}\]
    where $L_l: \R^{d_l} \to \R^{d_{l+1}}$, $d_1 = d_x$, $d_{L+1} = d_y$, and 
    \[L_l(x_l) = \sigma(A_lx_l + b_l), \quad x_l \in \R^{d_l},\]
    and $A_l \in \R^{d_{l+1} \times d_l}$ is a matrix and $b_l \in \R^{d_{l+1}}$ is a vector. The activation $\sigma: \R \to \R$ is a $C^s$, non-polynomial function for some integer $s \geq 0$,
    that is extended to act component-wise on vector inputs. \hfill$\lozenge$
\end{definition}

The following result is a well-known fact about the approximation capability of finite-dimensional feedforward neural networks. We restate it here for convenience and refer to \cite[Theorem 4.1]{pinkus1999approximation} for the proof.
\begin{lemma}\label{lem:NN-approx}
    Let $\Omega \subset \R^d$ be a bounded domain, and assume $\sigma \in C^s(\R)$ is not a polynomial.  
Let $u\in C^s(\Obar; \R^k)$ for some integer $s \geq 0$. Then,
for any non-polynomial activation $\sigma \in C^s$ as in Definition \ref{def:finite-NN}, and for any $\epsilon >0$, there exists a finite-dimensional neural network $\tilde{u}$ such that
    \begin{equation*}
        \|u - \tilde{u}\|_{C^s(\Obar; \R^k)}  = \sup_{|\alpha| \leq s} \sup_{x \in \Obar} |D^\alpha u(x) - D^\alpha \tilde{u}(x) | \leq \epsilon.
    \end{equation*}
\end{lemma}

\section{Proof of Proposition \ref{prop:cases}}\label{appd:prop-proof}
Here we provide a proof of Proposition \ref{prop:cases}, which details some familiar examples where the Smoothed Approximation Assumption holds. Throughout this proof, we refer to the three conditions of the Smoothed Approximation Assumption \ref{ass:smoothed_appr} by their enumerations \ref{it:smooth-conv}, \ref{it:proj-conv}, and \ref{it:cont-func}.
\begin{proof}
\textit{Case 1: } We note that this case is largely taken from the proof of \cite[Proposition A.6]{lanthaler2025nonlocality}. By \cite[Lemma A.1]{lanthaler2025nonlocality}, there exists a continuous extension mapping $\cE: C^{s'}(\Obar) \to C^{s'}_\text{per}(B)$. Then $\sfK'_{\text{per}} := \cE(\sfK')$ is a compact subset of $C^{s'}_{\text{per}}(B)$, and we can identify $B \simeq \Td$ in a canonical way. Fix $\delta >0$, and let $w'_\delta: B \to \R$ be the standard mollification (see Definition \ref{def:standard-moll}) of the periodic function $w': B \to \R$. We may then define the family of linear maps $\{f_\delta\}_{\delta >0}$ by the composition of the mollification and periodic extension operator; i.e., $f_\delta(w) := \cE(w)_\delta$. The standard mollifier satisfies \ref{it:smooth-conv}:
\begin{align*}
    \lim_{\delta \to 0 } \sup_{w \in \sfK'} \|f_\delta(w) - w \|_{C^{s'}(\Obar)}& = \lim_{\delta \to 0 } \sup_{w \in \sfK'} \|\cE(w)_\delta - w \|_{C^{s'}(\Obar)}  \\
    & = \lim_{\delta \to 0} \sup_{w' \in \sfK'_{\text{per}}} \|w'_\delta - w'\|_{C^{s'}(\Obar)}= 0. 
\end{align*}
 Since the extension operator and the mollifier are continuous, property \ref{it:cont-func}
is also satisfied. It follows from classical results in harmonic analysis  \cite[Chapter 1, Section 3]{jackson1930theory} 
that approximation by Fourier series on $C^{s'}(B)$ converges uniformly over the mollified set 
$\{w'_\delta | \; w' \in \sfK'_{\text{per}}\}$: 
\begin{equation*}
    \lim_{J\to \infty} \sup_{w'\in \sfK'_{\text{per}}} \Big \| w'_\delta- \sum_{j=1}^J \la w'_\delta, \eta_j \ra_{L^2(B)} \eta_j \Big\|_{C^{s'}_{\text{per}}(B)} = 0.
\end{equation*}
It can then be seen that 
\begin{align*}
    \lim_{J \to \infty}\sup_{w \in \sfK'}&\Big\|f_\delta(w) - \sum_{j=1}^J \la f_\delta(w), \eta_j \ra_{L^2(B)} \eta_j \Big\|_{C^{s'}(\Obar)}\\ & =  \lim_{J \to \infty}\sup_{w' \in \sfK'_{\text{per}}}\Big\|w'_\delta - \sum_{j=1}^J \la w'_\delta, \eta_j \ra_{L^2(B)} \eta_j \Big\|_{C^{s'}(\Obar)}\\
    & \leq  \lim_{J \to \infty}\sup_{w' \in \sfK'_{\text{per}}}\Big\|w'_\delta - \sum_{j=1}^J \la w'_\delta, \eta_j \ra_{L^2(B)} \eta_j \Big\|_{C^{s'}_{\text{per}}(B)} = 0,
\end{align*}
satisfying property \ref{it:proj-conv}.

\textit{Case 2:} 
In what follows, for any $\gamma>0$,
\begin{equation} \label{eq:og}
\Omega^{\gamma}:=\{x \in \Omega: \text{dist}(x, \p\Omega) \geq \gamma\},
\end{equation}
and $\mathbf{1}_{\Omega^{\gamma}}$ is the indicator function on $\Omega^{\gamma}$.
For any $\delta >0$, define the smooth cutoff mollification of $w \in \sfK'$ by 
\begin{equation*}
    f_\delta(w) = \rho_\delta \ast (\chi_{4\delta} w),
\end{equation*}
where $\rho_\delta$ is the standard mollifier as in Definition \ref{def:standard-moll}, and $\chi_\delta = \rho_{\delta/4} \ast \mathbf{1}_{\Omega^{3\delta/4}}$.
Then $\chi_\delta$ satisfies
\begin{equation*}
    \chi_\delta(x) = \begin{cases}
        1 \quad \text{ if } \text{dist}(x, \partial\Omega) \geq \delta, \; x \in \Omega\\
        0 \quad \text{ if } \text{dist}(x, \partial\Omega) \leq \frac{\delta}{2}, x \in \Omega, \\
        0 \quad \text{ if } x \notin \Omega,
    \end{cases}
\end{equation*}
with $0 \leq \chi_\delta \leq 1$ and $\|\nabla \chi_\delta \|_{\infty} \lesssim \delta^{-1}$.
Then $\chi_{4\delta} w$ is continuous and equal to $0$ within $2\delta$ of the boundary $\p\Omega$. We show first that property \ref{it:smooth-conv} holds. 

Recall that since $\sfK'$ is compact, it is equicontinuous; i.e.~there exists a modulus of continuity $\gamma$ such that for all $w\in \sfK'$, it holds that $|w(x) - w(y)| \leq \gamma(|x-y|)$ for $x,y \in \Obar$, and that $\gamma(t) \to 0$ as $t \to 0^+$. We bound
\begin{align*}
    \lim_{\delta \to 0} \sup_{w \in \sfK'} \|f_\delta(w) - w\|_{C(\Obar)} \leq \lim_{\delta \to 0}\sup_{w \in \sfK'} \Big(\underbrace{\|f_\delta(w) - \chi_{4\delta}w\|_{C(\Obar)}}_{\text{(I)}} + \underbrace{\|\chi_{4\delta}w - w\|_{C(\Obar)}}_{\text{(II)}}\Big).
\end{align*}
We bound components (I) and (II) separately. For any $v \in \{\chi_{4\delta}w: \; w \in \sfK'\}$ and $x,y \in \Obar$, we have
\begin{align}
    |v(x) - v(y)| & \leq \sup_{w \in \sfK'} |(\chi_{4\delta}w)(x) - (\chi_{4\delta}w)(y)|\notag\\
    & \leq \sup_{w \in \sfK'} |\chi_{4\delta}(x)| |w(x) - w(y)| + |w(y)| |\chi_{4\delta}(x) - \chi_{4\delta}(y)| \notag\\
    & \leq \gamma(|x-y|) + |w(y)|. \label{eqn:Kpdelta-MOC}
\end{align}
We are now prepared to bound component (I): 
\begin{align*}
    \lim_{\delta \to 0}&\sup_{w \in \sfK'}\|\rho_\delta \ast (\chi_{4\delta}w) - \chi_{4\delta}w\|_{C(\Obar)} \\&= \lim_{\delta \to 0}\sup_{w\in \sfK'}\sup_{x \in \Obar} \Big| \int_{B_\delta(0)} \rho_\delta(y) ((\chi_{4\delta} w)(x-y) - (\chi_{4\delta}w)(x)) \dy\Big|\\
    & \leq \lim_{\delta \to 0} \sup_{w \in \sfK'}\sup_{x \in \Obar} \sup_{|y| \leq \delta }|(\chi_{4\delta}w)(x-y) - (\chi_{4\delta}w)(x)|\\
    & = \lim_{\delta \to 0} \sup_{w\in\sfK'}\max\Big\{\sup_{\substack{x \in \Omega^{5\delta}\\|y| \leq \delta}}  |w(x-y) - w(x) |, \sup_{\substack{x \in \Obar \setminus \Omega^{5\delta}\\ |y| \leq \delta}}  |(\chi_{4\delta}w)(x-y) - (\chi_{4\delta}w)(x)|\Big\}\\
    & \leq \lim_{\delta \to 0} \max\Big\{ \gamma(\delta), \gamma(\delta) + \sup_{w \in \sfK'} \sup_{x \in \Obar\setminus\Omega^{5\delta}} |w(x)|\Big\}
\end{align*}
by \eqref{eqn:Kpdelta-MOC}. Then, 
\begin{align*}
     \lim_{\delta \to 0}&\sup_{w \in \sfK'}\|\rho_\delta \ast (\chi_{4\delta}w) - \chi_{4\delta}w\|_{C(\Obar)} \leq \lim_{\delta \to 0} \gamma(\delta) + \gamma(5\delta)  = 0,
\end{align*}
since $w = 0$ on $\p\Omega$, and any $x \in \Obar \setminus\Omega^{5\delta}$ is at a distance at most $5\delta$ from the boundary. Bounding component (II) is simpler, as we see that
\begin{align*}
    \lim_{\delta \to 0}\sup_{w \in \sfK'} \|\chi_{4\delta}w -w \|_{C(\Obar)} & \leq \lim_{\delta \to 0} \sup_{w \in \sfK'}\sup_{x \in \Obar \setminus \Omega^{4\delta}}|\chi_{4\delta}(x) w(x) - w(x) | \\
    & \leq \lim_{\delta \to 0}\sup_{w \in \sfK'} \sup_{x \in \Obar \setminus \Omega^{4\delta}} |w(x)|\\
    & \leq \lim_{\delta \to 0} \gamma(|4\delta|) = 0. 
\end{align*}
Thus property \ref{it:smooth-conv} holds.
It is clear that property \ref{it:cont-func} also holds since $\chi_\delta$ is continuous and the eigenfunctions of the Laplacian are $C^\infty(\Omega)$. For property \ref{it:proj-conv}, set $B = \Obar$ and note that for any fixed $\delta >0$, $f_\delta$ is a bounded, continuous map from $C(\Obar)$ to $C_c^k(\Obar)$ for any positive integer $k$, where $C^k_c(\Obar)$ denotes the subspace of functions in $C^k(\Obar)$ with compact support in $\Omega$ and thus are $0$ on the boundary. Then the image of $\sfK'$ under $f_\delta$, denoted $\sfK'_\delta$, is compact in $C^k_c(\Obar)$. Applying Lemma \ref{lem:eigfunc-proj}, we obtain convergence of the projection of $\sfK'_\delta$ onto the eigenfunctions of the Laplacian in $C(\Obar)$, which achieves  property \ref{it:proj-conv}.

\textit{Case 3:} 
Let $\{P_n\}_{n=0}^\infty$ denote the Legendre polynomials, $P_n$ of degree $n$, 
orthonormalized so that $\bigl(n+\frac12\bigr)\int_{-1}^1 P_m(x) P_n(x) \dx = \delta_{mn}$,
the standard Kronecker delta \cite[Equation 1.3.19]{shen2011spectral}.
We tensorize the Legendre polynomials in $d$ dimensions in the following manner: for $\nu \in \N_0^d$, let $\Phi_\nu(x) = \prod_{i=1}^d P_{\nu_i}(x_i)$. The set $\{\Phi_\nu\}_{\nu \in \N_0^d}$ is then an orthonormal basis of $L^2(\Omega)$. Let $B'$ be a bounded hypercube containing $\Omega$ in its interior. There exists a continuous, linear operator $\cE: C^{s'}(\Omega) \to C^{s'}(B')$ such that $\cE(u)|_{\Omega} = u$ \cite[Lemma A.1]{lanthaler2025nonlocality}. Let $\rho_\delta$ be the standard mollifier in Definition \ref{def:standard-moll}, and define $f_\delta(w) = (\rho_\delta\ast \cE(w))|_{\Omega}$ for $w \in C^{s'}(\Omega)$. We assume $\delta < \text{dist}(\Omega, \p B')$ so that the mollifier is defined on $\Omega$. To prove property \ref{it:smooth-conv}, observe that for $|\alpha| \leq s'$, $D^\alpha f_\delta(w) - D^\alpha w= \rho_\delta \ast (D^\alpha \cE(w)) - D^\alpha \cE(w))$ on $\Omega$. Since $\sfK'$ is compact in $C^{s'}(\Omega)$ and $\cE$ is continuous, the set $\{D^\alpha \cE(w): \; w \in \sfK'\}$ is compact for any choice of $|\alpha| \leq s'$ and thus is uniformly equicontinuous; denote by $\omega_\alpha$ the modulus of continuity of this set. Since the set of $\alpha$ such that $|\alpha| \leq s'$ is finite, it follows that 
\begin{align*}
\lim_{\delta \to 0}&\max_{\alpha: |\alpha| \leq s'} \sup_{w \in \sfK'} \|D^\alpha f_\delta(w) - D^\alpha w \|_{L^\infty(\Omega)} \\& = \lim_{\delta \to 0}\max_{\alpha: |\alpha| \leq s'} \sup_{w \in \sfK'} \|\rho_\delta \ast (D^\alpha \cE(w)) - D^\alpha \cE(w)\|_{L^\infty(\Omega)} \\
& \leq \lim_{\delta \to 0} \max_{\alpha: |\alpha| \leq s'}  \omega_\alpha(\delta) = 0,
\end{align*}
so property \ref{it:smooth-conv} holds. Property \ref{it:cont-func} holds as well since $\rho_\delta \ast \cE$ is a bounded linear function from $C^{s'}(\Omega)$ to $L^2(\Omega)$, and each $\Phi_\nu$ satisfies $\|\Phi_\nu\|_{L^2} = 1$ so $w \mapsto \la f_\delta(w), \Phi_\nu \ra$ is a bounded linear functional and thus a continuous functional. 

We proceed with proving property \ref{it:proj-conv}:
\begin{align*}
    \sup_{w \in \sfK'}& \big\|f_\delta(w) - \sum_{\|\nu\|_{\infty} \leq J} \la f_\delta(w), \Phi_\nu \ra \Phi_\nu \big\|_{C^{s'}(\Omega)}  \\
    &= \sup_{w\in \sfK'} \big\|\sum_{\|\nu\|_{\infty} > J} \la f_\delta(w), \Phi_\nu \ra \Phi_\nu \big\|_{C^{s'}(\Omega)}\\
    & \leq \sup_{w\in \sfK'} \sum_{\|\nu\|_\infty > J} | \la f_\delta(w), \Phi_\nu\ra | \big \| \prod_{i=1}^d P_{\nu_i} \big\|_{C^{s'}(\Omega)}\\
    & \leq \sup_{w\in \sfK'} \sum_{\|\nu\|_\infty > J}  | \la f_\delta(w), \Phi_\nu \ra |\prod_{i =1, \nu_i \neq 0}^d \|P_{\nu_i}\|_{C^{s'}([-1,1])},
\end{align*}
where in the last line we have used that $D^\alpha \Phi_\nu = \prod_{i=1}^d P_{\nu_i}^{(\alpha_i)}(x_i)$
together with the fact that $\|P_0\|_{C^{s'}([-1,1])} = \frac{1}{\sqrt{2}} < 1$, which allows the
factors with $\nu_i = 0$ to be discarded. Due to our chosen rescaling, $\|P_n\|_{L^\infty([-1,1])} \lesssim n^{1/2}$,
it follows from the Markov brothers' inequality that $\|P_{\nu_i}\|_{C^{s'}([-1,1])} \leq C (\nu_i)^{2s' + \frac{1}{2}}$ for some constant $C$ independent of $s'$; without loss of generality, we assume $C \geq 1$. We are then left with 
 \begin{subequations}\label{eqn:leg-prop-2}
 \begin{align}
     \sup_{w \in \sfK'}& \big\|f_\delta(w) - \sum_{\|\nu\|_{\infty} \leq J} \la f_\delta(w), \Phi_\nu \ra \Phi_\nu \big\|_{C^{s'}(\Omega)}\\ & \leq C^d \underbrace{\sup_{w\in \sfK'} \sum_{\|\nu\|_\infty > J} | \la f_\delta(w), \Phi_\nu \ra |\prod_{i =1, \nu_i \neq 0}^d  (\nu_i)^{2s' + \frac{1}{2}}}_{\text{(I)}}.
 \end{align}  
  \end{subequations}
  We introduce the Legendre operator \[\cL_i = -\p_{x_i}((1-x_i^2)\p_{x_i}),\] which is self-adjoint on $L^2(\Omega)$ with zero boundary terms since $(1-x_i^2)$ vanishes at $x_i = \pm 1$, and it holds that for $P_n$ a polynomial in $x_i$, $\cL_i P_n = n(n+1) P_n$ \cite[Section 1.3]{tang2006spectral}. Also note that the $\cL_i$ commute. This allows us to write $\Phi_\nu = \left(\prod_{i=1, \nu_i \neq 0}^d \frac{\cL^k_i}{(\nu_i(\nu_i + 1))^k}\right)\Phi_\nu$. We may then bound quantity (I) by
  \begin{align}
      \sup_{w\in \sfK'}& \sum_{\|\nu\|_\infty > J} | \la f_\delta(w), \Phi_\nu \ra |\prod_{i =1, \nu_i \neq 0}^d  (\nu_i)^{2s' + \frac{1}{2}}\notag \\ &\leq  \sup_{w\in \sfK'}\sum_{\|\nu\|_\infty > J}  |\la \big(\prod_{i=1, \nu_i \neq 0}^d \cL_i^k \big)f_\delta(w), \Phi_\nu \ra | \prod_{i=1, \nu_i \neq 0}^d(\nu_i)^{2s' + \frac{1}{2} - 2k}\notag \\
      & \leq \sup_{w \in \sfK'} \sum_{\|\nu\|_\infty > J} \|\big(\prod_{i=1, \nu_i \neq 0}^d\cL_i^k\big) f_\delta(w) \|\prod_{i = 1, \nu_i \neq 0}^d (\nu_i)^{2s' + \frac{1}{2} - 2k}.\label{eqn:leg-prop2-I}
  \end{align}
  In the above, we have used the fact that $(\nu_i (\nu_i + 1)) \geq \nu_i^2$ and that $\|\Phi_\nu\| = 1$. Since $\cL_i f_\delta(w) = -2x_i \p_{x_i} f_\delta(w) + (1-x_i^2) \p^2_{x_i} f_\delta(w)$ and $x_i \in [-1,1]$, the operator $\prod_{i=1, \nu_i \neq 0}^d \cL^k_i$ is a differential operator of order at most $2kd$ with coefficients bounded on $\Omega$, so we can bound independently of $\nu$,  \[\|\left(\prod_{i=1, \nu_i \neq 0}^d \cL_i^k\right) f_\delta(w) \| \leq C_{k,d} \|f_\delta(w)\|_{C^{2kd}(\Omega)} \leq C_{k,d} C_{k,d,\delta}
  \sup_{w \in \sfK'} \|w\|_{C^{s'}(\Omega)},\] which is bounded in the final inequality for each fixed $\delta$ since $\sfK'$
  is compact in $C^{s'}(\Omega)$. Bound \eqref{eqn:leg-prop2-I} then becomes
  \begin{align}\label{eqn:leg-props2-II}
      C_{k,d}C_{k, d, \delta} \sup_{w \in \sfK'}\|w\|_{C^{s'}(\Omega)} \underbrace{\sum_{\|\nu\|_\infty > J} \prod_{i=1, \nu_i \neq 0}^d (\nu_i)^{2s' + \frac{1}{2} - 2k}}_{\text{(II)}}.
  \end{align}
  Let $b_{\nu_i}= (\nu_i)^{2s' + \frac{1}{2} -2k}$ if $\nu_i > 0$, and $b_0 = 1$ otherwise. Let $S = \sum_{n=1}^\infty n^{2s' + \frac{1}{2} - 2k}$ and $S_J = \sum_{n=1}^J n^{2s' + \frac{1}{2} - 2k}$. Assume integer $k > s' + \frac{3}{4}$ so that $S$ converges. Finally, note that the following sum-product factors as $\sum_{\nu \in \N_0^d}\prod_{i=1}^d b_{\nu_i} = (1+S)^d$. As a result, we may rewrite (II) and bound as
  \begin{align*}
      \sum_{\|\nu\|_\infty > J} \prod_{i=1, \nu_i \neq 0}^d (\nu_i)^{2s' + \frac{1}{2} - 2k} & = (1+S)^d - (1+S_J)^d\\
      & = d\xi^{d-1} (S-S_J) \\
      & = d\xi^{d-1} \sum_{n > J} n^{2s' + \frac{1}{2} -2k} \\
      & \leq C_{d, k, s'}J^{\frac{3}{2} +2s' - 2k},
  \end{align*}
  where we have used the mean value theorem in the second equality for some $\xi \in (1+S_J, 1+S)$.
  Combining the above bound with \eqref{eqn:leg-prop-2}, \eqref{eqn:leg-prop2-I}, and \eqref{eqn:leg-props2-II}, we find that 
  \begin{align*}
      \sup_{w \in \sfK'}& \big\|f_\delta(w) - \sum_{\|\nu\|_{\infty} \leq J} \la f_\delta(w), \Phi_\nu \ra \Phi_\nu \big\|_{C^{s'}(\Omega)} \\ &\leq C_{k,d,\delta,s'}\sup_{w\in \sfK'}\|w\|_{C^{s'}(\Omega)} J^{\frac{3}{2} +2s' -2k}.
  \end{align*}
  Since $\delta$ is fixed and $k > s' + \frac{3}{4}$, the bound tends to $0$ as $J \to \infty$, proving \ref{it:proj-conv}.
  \end{proof}

\section{Proof of Lemma \ref{lem:decomp}} \label{appd:lem-decomp}

Before proceeding with the proof, we repeat the lemma statement for ease of readability.

\lemdecomp*

\begin{proof}[Proof of Lemma \ref{lem:decomp}]
    \indent\textit{Step 1: Approximation by $\{\eta_j\}_{j=1}^\infty$.} 
    Without loss of generality, we may assume $k'=1$, as we can always approximate $C^{s'}(\overline{\Omega};\R^{k'}) = [C^{s'}(\overline{\Omega};\R)]^{k'}$ component-wise. Thus, denote by $C^{s'}(\overline{\Omega})$ the space $C^{s'}(\overline{\Omega};\R)$. 

    Let $\{f_\delta\}_{\delta >0}$ be the family of maps in Assumptions \ref{ass:smoothed_appr}, and $B$ the domain in Assumptions \ref{ass:smoothed_appr} item \ref{it:proj-conv}. Under these assumptions, we have that 
    \begin{equation*}
        \lim_{J \to \infty} \sup_{w \in \mathsf{K}'} \Big\|f_\delta(w) - \sum_{j=1}^J \la f_\delta(w), \eta_j \ra_{L^2(B)}\eta_j\Big\|_{C^{s'}(\Omega)}=0,
    \end{equation*}
    and $\lim_{\delta \to 0} \sup_{w \in \sfK'}\|f_\delta(w) - w\|_{C^{s'}(\Omega)}=0$.
     In particular, given $\epsilon >0$, we can first find $\delta = \delta(\epsilon) >0$ such that $\sup_{w \in \mathsf{K}'}\|f_\delta(w) - w\|_{C^{s'}(\Omega)}  \leq \epsilon/2$, and then $J \in \N$ such that $\sup_{w \in \mfK'}\|f_\delta(w) - \sum_{j=1}^J\la f_\delta(w), \eta_j\ra_{L^2(B)}\eta_j\|_{C^{s'}(\Omega)} \leq \epsilon/2$. It follows from the triangle inequality that
    \begin{equation}\label{eqn:tri-inq-result}
     \sup_{w \in \mfK'} \Big\|w - \sum_{j=1}^J \la f_\delta(w), \eta_j \ra_{L^2(B)} \eta_j\Big\|_{C^{s'}(\Omega)} \leq \epsilon.
     \end{equation}
      As a consequence, it is clear that 
    \begin{align*}
        &\sup_{u \in \mfK} \Big\| \Psi^\dagger(u) - \sum_{j=1}^J \la f_\delta(\Psi^\dagger(u)), \eta_{j}\ra_{L^2(B)}\eta_j \Big\|_{C^{s'}(\Omega)} \\
        & = \sup_{w \in \mfK'} \Big \| w - \sum_{j=1}^J \la f_\delta(w), \eta_j\ra_{L^2(B)} \eta_j \Big \|_{C^{s'}(\Omega)} \leq \epsilon. 
    \end{align*}  
Defining $\beta_j(u) = \la f_\delta(\Psi^\dagger(u)), \eta_j \ra_{L^2(B)}$ we have that $\tilde{\Psi}^\dagger(u) = \sum_{j=1}^J \beta_j(u)\eta_j$ satisfies 
    \begin{equation}\label{eqn:step1-result}
    \sup_{u \in \sfK} \|\Psi^\dagger(u) - \tilde{\Psi}^\dagger(u)\|_{C^{s'}(\Omega)} \leq \epsilon.
    \end{equation}
\indent\textit{Step 2: Construction of $\alpha_j$.} The above inequality is almost the desired result, but $\beta_j$ does not define a continuous functional $L^1(\Omega; \R^k) \to \R$. The remainder of the proof is very similar to that of the proof of \cite[Proposition A.6]{lanthaler2025nonlocality}; we include it here for convenience.
By \cite[Lemma A.2]{lanthaler2025nonlocality}, there exists a continuous operator $\cM_\delta : L^1(\Omega; \R^k) \to C^s(\Obar; \R^k)$ such that over the compact set $\mfK \subset C^s(\Obar; \R^k)$, we have 
\[
\sup_{u \in \mfK}\|u - \cM_\delta u\|_{C^s(\Obar)} \to 0
\]
as $\delta \to 0$. Our goal is to define $\alpha_j: L^1(\Omega; \R^k) \to \R$ by $\alpha_j(u) = \beta_j(\cM_\delta u)$ for suitably chosen $\delta >0$. Recall that, for fixed $\delta_0 > 0$, the set $\mfK_\delta$ defined by $\mfK_\delta = \cup_{0 \leq \delta \leq \delta_0} \cM_\delta(\mfK)$ is a compact subset of $C^s(\Obar;\R^k)$  by \cite[Lemma A.4]{lanthaler2025nonlocality}. Note that for any $j = 1, \dots, J$, by Assumptions \ref{ass:smoothed_appr} property \ref{it:cont-func},together with the continuity of $\Psi^\dagger$, we have a continuous mapping $\beta_j: \mfK_\delta \to \R$. Since $\mfK_{\delta}$ is compact, it follows that there exists a modulus of continuity $\omega: [0,\infty) \to [0, \infty)$, continuous in its argument
and  satisfying $\omega(0) = 0$, such that 
\[
|\beta_j(u) - \beta_j(u')| \leq \omega(\|u - u'\|_{C^s(\Obar)}), \quad \forall u, u' \in \mfK_\delta,
\]
holds for all $j = 1, \dots, J$. It follows that 
\[
|\beta_j(u) - \beta_j(\cM_\delta u)| \leq \omega(\|u - \cM_\delta u) \|_{C^s(\Obar)}), 
\]
for any $0 < \delta \leq \delta_0$. By \cite[Lemma A.2]{lanthaler2025nonlocality}, $\cM_\delta u \to u$ converges uniformly over $u \in \mfK$, so that we may conclude 
\[
\lim_{\delta \to 0} \sup_{u \in \mfK} |\beta_j(u) - \beta_j(\cM_\delta u)| \leq \lim_{\delta \to 0} \omega(\sup_{u \in \mfK} \|u - \cM_\delta u\|_{C^s(\Obar)}) = 0. 
\]
In particular, we can choose $\delta > 0$ sufficiently small, to ensure that $\alpha_j(u):= \beta_j(\cM_\delta u)$ satisfies 
\begin{equation} \label{eqn:step2-result}
\sup_{u \in \mfK} |\beta_j(u) -\alpha_j(u)| \leq \frac{\epsilon}{J\max_{j = 1, \dots, J} \|\eta_j\|_{C^{s'}(\Omega)}}
\end{equation}
for all $j = 1, \dots, J$. Since $\delta >0$, we also note that $\cM_\delta: L^1(\Omega; \R^k) \to C^s(\Obar; \R^k)$ is a continuous mapping, and hence $\alpha_j = \beta_j \circ \cM_{\delta}$ is continuous as a mapping $\alpha_j: L^1(\Omega; \R^k) \to \R$. 

\paragraph{Step 3: (Conclusion)} Combining the results of \eqref{eqn:step1-result} and \eqref{eqn:step2-result}, we obtain 
\begin{align*}
    \sup_{u \in \mfK} \Big \| \Psi^\dagger(u)& - \sum_{j=1}^J \alpha_j(u) \eta_j \Big \|_{C^{s'}(\Omega)}\\ &\leq \sup_{u \in \mfK} \Big\| \Psi^\dagger(u) - \sum_{j=1}^J \beta_j(u) \eta_j \Big\|_{C^{s'}(\Omega)} + \sup_{u \in \mfK} \Big \| \sum_{j=1}^J (\beta_j(u) - \alpha_j(u)) \eta_j \Big \|_{C^{s'}(\Omega)}\\
    & \leq \sup_{u \in \mfK} \Big\| \Psi^\dagger(u) - \sum_{j=1}^J \beta_j(u) \eta_j \Big\|_{C^{s'}(\Omega)} \\&+ J \max_{j = 1, \dots, J} \|\eta_j\|_{C^{s'}(\Omega)} \max_{j = 1, \dots, J} \sup_{u \in \mfK} |\beta_j(u) - \alpha_j(u)| \\
    & \leq 2\epsilon.
\end{align*}
Since $\epsilon >0$ was arbitrary, this proves the claim. 
\end{proof}

\section{Proofs of Lemma \ref{lemma:alpha-approx-single} and \ref{lemma:alpha-approx}}\label{appd:alpha-approx}

Before we proceed with the proofs, we will need the following lemma, which allows us to approximate certain integrals by the integral of a finite-dimensional neural network weighted by a positive function.

\begin{lemma} \label{lem:int-approx}
    Let $\Omega$ be a bounded domain in $\Rd$, and let $\mathsf{K}$ be compact in $L^1(\Omega; \R^k)$ and assume that
    $M_{\mathsf{K}}: = \sup_{u \in \mathsf{K}} \|u \|_{L^\infty}<\infty.$
    Let $\xi \in C(\Obar;\R^k) \cap L^1(\Omega;\R^k)$ be given and fixed. Furthermore, let $\psi \in C(\Omega;\R_{\geq 0})$ be such that $\sup_{y \in \Omega} |\psi(y)| < \infty$
    and such that $\psi$ is $0$ only on a set of Lebesgue measure $0$.  Then for any $\epsilon >0$ there exists a neural network $\tilde{R}:\R^k \times \Obar \to \R$ such that 
    \[
    \sup_{u \in \sfK} \left| \int_{\Omega} \tilde{R}(u(y), y) \psi(y) \dy - \int_{\Omega} u(y) \cdot \xi(y) \dy \right| < \epsilon. 
    \]
\end{lemma}

\begin{proof}
Since $\psi \geq 0$ on $\Omega$ and equals $0$ only on a set of measure $0$, we may assume without loss of generality that $\int_\Omega \psi(y) \dy = 1$ since the final layer of $\tilde{R}$ may be rescaled to account for the rescaling of $\psi$.  Define $\Omega^\gamma := \{y \in \Omega: \psi(y) < \gamma\}$ and fix $\gamma$ such that $M_\sfK \|\xi\|_{C(\Obar)} |\Omega^\gamma| < \frac{\epsilon}{4}$. Then define
\begin{equation*}
    \psi^\gamma(y) = \max(\psi(y), \gamma), \quad y \in \Omega.
\end{equation*}
It holds that $\psi^\gamma \in C(\Omega)$. 
  Choose a compact $F \subset \Omega$, contained in the interior of $\Omega$ if $\Omega$ is closed, such that $M_\sfK \|\xi\|_{C(\Obar)}|\Omega \setminus F| < \frac{\epsilon}{4}$.
  Thus, $F$ does not touch $\partial \Omega$, which enables the cutoff function $\chi$, that we now go on to define, to have space to deform continuously to $0$.  Define $\chi \in C(\Obar; [0,1])$ such that $\chi = 1$ on $F$ and $0$ on $\p \Omega$.
  Then set
\begin{equation*}
    g(y)= \begin{cases}
        \frac{\chi(y)}{\psi^\gamma(y)}, \quad y \in \Omega, \\
        0 , \quad y \in \p\Omega.
    \end{cases}
\end{equation*}
Since $\psi^\gamma$ is positive everywhere on $\Omega$ and $\chi$ is $0$ on the boundary, $g$ is continuous on $\Obar$, and  thus the function $R: (a,y)\mapsto (a\cdot \xi(y))g(y)$ is continuous on domain $[-M_{\mathsf{K}}, M_{\mathsf{K}}]^k \times \Obar$. By universality of ordinary neural networks, i.e.~Lemma \ref{lem:NN-approx}, there exists a network $\tilde{R}$ such that 
\begin{equation*}
    \sup_{a \in [-M_{\mathsf{K}}, M_{\mathsf{K}}]^k, y \in \Obar} |R(a,y) - \tilde{R}(a,y)| < \frac{\epsilon}{2}.
\end{equation*}
We proceed with the bound.
\begin{align*}
&\sup_{u \in \mathsf{K}} \left| \int_\Omega \tilde{R}(u(y),y) \psi(y) \dy - \int_\Omega u(y) \cdot \xi(y) \dy \right|\\
& \leq \sup_{u \in \mathsf{K}} \int_\Omega \left|\tilde{R}(u(y),y) - R(u(y),y)\right| \psi(y) \dy  + \left| \int_\Omega R(u(y),y) \psi(y) - u(y) \cdot \xi(y) \dy\right| \\
& \leq \frac{\epsilon}{2}   + \sup_{u \in \mathsf{K}}\left| \int_\Omega u(y)\cdot \xi(y) \left(\frac{\chi(y)\psi(y)}{\psi^\gamma(y)}- 1\right) \dy\right|,
\end{align*}
where we have used the normalization $\int_\Omega \psi(y) \dy = 1$.
Since $\frac{\psi(y)}{\psi^\gamma(y)} \leq 1$ and $\chi \leq 1$, the expression $\frac{\chi(y)\psi(y)}{\psi^\gamma(y)} \in [0,1]$. Additionally, $\frac{\chi(y)\psi(y)}{\psi^\gamma(y)} = 1$ on $F \setminus \Omega^\gamma$, so we have $\left|\frac{\chi(y)\psi(y)}{\psi^\gamma(y)} - 1 \right| \leq \mathbf{1}_{\Omega \setminus F} + \mathbf{1}_{\Omega^\gamma}$. 
Therefore, 
\begin{align*}
    \sup_{u \in \mathsf{K}}\left|\int_{\Omega} u(y) \cdot \xi(y) \left(\frac{\chi(y)\psi(y)}{\psi^\gamma(y)} - 1 \right) \dy\right| & \leq M_{\mathsf{K}}\|\xi\|_{C(\Obar)}(|\Omega \setminus F| + |\Omega^\gamma|).
\end{align*}

By our choices of $\gamma$ and $F$,
\begin{equation*}
    \sup_{u \in \mathsf{K}} \left| \int_\Omega \tilde{R}(u(y),y) \psi(y) \dy - \int_\Omega u(y) \cdot \xi(y) \dy \right| < \frac{\epsilon}{2} + \frac{\epsilon}{4} + \frac{\epsilon}{4} = \epsilon.
\end{equation*}
\end{proof}

We also point out the following fact about continuous, nonzero functions.

\begin{lemma}\label{lem:Gamma}
Let $\Omega$ be a bounded set in $\Rd$. Assume that $\eta$ is continuous on $\Omega$ and not identically $0$. Then for any $\epsilon > 0$, there exists an open ball $\Gamma \subset \Omega$ and constant $\gamma$ such that 
    \[\|\gamma\eta - \mathbf{1}\|_{L^\infty(\Gamma)} < \epsilon.\]
\end{lemma}
\begin{proof}
    The result follows directly from the continuity of $\eta$.
\end{proof} 

Before proceeding with the proof of Lemma \ref{lemma:alpha-approx}, we define the standard mollifier. This definition is taken from the standard text \cite[Appendix C.4]{evans2022partial}.

\begin{definition}[Standard Mollifier]\label{def:standard-moll}
    Let $\Omega$ be a bounded Lipschitz domain, and define $\Omega^\epsilon$ as in \eqref{eq:og}. We say that $w_\epsilon: \Omega^\epsilon \to \R$ is the \emph{standard mollification} of 
$w \in L^1(\Omega)$, if \begin{equation}\label{eqn:standard-moll}
w_\epsilon := \rho_\epsilon \ast w, \; \text{ in } \Omega^\epsilon \quad \text{ for } \quad \rho_\epsilon(x) := \frac{1}{\epsilon ^d}\rho\left(\frac{x}{\epsilon}\right),
\end{equation}
where
\[\rho(x) := \begin{cases}
    C\exp\left(\frac{1}{|x|^2 -1}\right), \quad &\text{if } |x| < 1 \\
    0, \quad &\text{if } |x| \geq 1.
\end{cases}\]
is referred to as the \textit{standard mollifier}. \hfill$\lozenge$
\end{definition}

We now restate and prove the following two lemmas, stated originally in Section \ref{ssec:NNO-approx}. Since the proofs begin in the same way, we initiate the proofs together and make a note when they diverge. 

\lemalphaapproxsingle*

\lemalphaapprox*

\begin{proof}[Proof of Lemmas \ref{lemma:alpha-approx-single} and \ref{lemma:alpha-approx}]
        We can identify any $u \in L^1(\Omega;\R^k)$ with a function in $L^1(B;\R^k)$ for any set $B \supset \Omega$ via an extension of $u(x) = 0$ for $x \in B\setminus \Omega$. Then the compact subset $\sfK \subset L^1(\Omega;\R^k)$ can be identified with a compact subset of $L^1(B;\R^k)$. Let $\rho_\delta$ for $\delta > 0$ be the standard mollifier of Definition \ref{def:standard-moll}. We denote by $u_\delta(x) = (u \ast \rho_\delta)(x), \quad x \in \Omega,$ the mollification of $u$ extended to $B$ by $0$ outside of $\Omega$. Since $\sfK \subset L^1(\Omega;\R^k)$ is compact, 
    \[\lim_{\delta \to 0} \sup_{u \in \sfK}\|u - u_{\delta}\|_{L^1(\Omega;\R^k)} = 0.\]
    Since $\alpha:L^1(\Omega;\R^k)\to\R$ is continuous, it follows that $\alpha_{\delta}:L^1(\Omega;\R^k)\to\R$ defined by $\alpha_\delta(u) = \alpha(u_\delta)$ for $\delta >0$, converges uniformly over $\sfK$ as $\delta \to 0$: 
    \begin{equation}\label{eqn:alpha-limsup}
    \lim_{\delta \to 0} \sup_{u \in \sfK} |\alpha(u) - \alpha_{\delta}(u) |\leq \lim_{\delta \to 0} \sup_{u \in \sfK} |\alpha(u) - \alpha(u_\delta)| =0.
    \end{equation}
Let $\{\xi_j\}_{j=1}^\infty$ be any orthonormal basis of $L^2(\Omega;\R^k)$, which we may additionally choose to be smooth; $\xi_j \in C^\infty(\Omega;\R^k)$. Since $L^2(\Omega) \hookrightarrow L^1(\Omega)$ and $\{u_\delta: \; u \in \sfK\}$ is compact in $L^2(\Omega)$, due to being the image of compact $\sfK$ in $L^1$ under the continuous map of the mollification,  it holds that 
    \begin{equation}\label{eqn:alpha-projection}
        \lim_{J \to \infty}\sup_{u \in \sfK} \|u_\delta - \sum_{j=1}^J \la u_\delta, \xi_j \ra\xi_j\|_{L^1(\Omega;\R^k)} = 0.
    \end{equation}
  
    Then define
    \begin{equation*}
        \alpha_{\delta, J}(u) = \alpha\left(\sum_{j=1}^J \la u_\delta, \xi_j \ra \xi_j\right). 
    \end{equation*}
    Furthermore, using evenness of the mollifier, we have the identity
   \begin{subequations}\label{eqn:conv-identity}
    \begin{align}
        \la u_\delta, \xi_j \ra &= \int_{\Omega} u_\delta(y) \cdot \xi_j(y) \dy = \int_{\R^d} (u \ast \rho_\delta)(y) \cdot \xi_j(y) \dy \\
        & = \int_{\R^d} u(y) \cdot (\xi_j \ast \rho_{\delta}) \dy = \int_{\Omega} u(y) \cdot \xi_{j, \delta}(y)  \dy
    \end{align}
    \end{subequations}
    with the convention that $u$ and $\xi_j$ are expanded by $0$ outside of $\Omega$, and 
    \begin{align}
       \xi_{j, \delta}(y) = (\xi_j \ast \rho_\delta)(y), \quad y \in \Obar.
    \end{align}
    Thus, we may rewrite 
    \begin{equation}\label{eqn:alpha-delta-J}
     \alpha_{\delta, J}(u) = \alpha \Big(\sum_{j=1}^J \la u, \xi_{j, \delta} \ra \xi_j\Big)
    \end{equation}
    From \eqref{eqn:alpha-limsup}, \eqref{eqn:alpha-projection}, and continuity of $\alpha$, for any $\epsilon_\alpha > 0$, it holds that for $\delta= \delta(\sfK) > 0$ small enough and $J = J(\delta, \sfK) \in \N$ large enough, 
    \begin{equation}\label{eqn:alpha-convergence}
        \sup_{u \in \mathsf{K}}|\alpha(u) - \alpha_{\delta, J}(u)| \leq \epsilon_\alpha. 
    \end{equation}
    
    We will thus seek to approximate $\alpha_{\delta, J}(u)\eta$ by the NNO architecture. Recall that $\eta_1$ is the single basis function satisfying properties \ref{enum:eta1-1}--\ref{enum:eta1-4}, and thus satisfies the conditions on $\psi$ in Lemma \ref{lem:int-approx}. Additionally, $\xi_{j, \delta} \in C(\Obar; \R^k) \cap L^1(\Omega; \R^k)$ and thus satisfies the conditions on $\xi$ for the same lemma. Parallelizing the networks constructed in Lemma \ref{lem:int-approx}, for any $\epsilon' >0$, there exists a neural network $R': \R^k \times \Obar \to \R^J$, $(v,x) \mapsto R'(v,x) = \left(R_1(v,x), \dots, R_J(v,x)\right)$ such that 
    \begin{equation}\label{eqn:R-approx}
    \sup_{u \in \sfK} \left|\int_{\Omega} R_j (u(y), y)\eta_1(y) \dy - \int_{\Omega} u(y) \cdot \xi_{j,\delta}(y) \dy\right| < \epsilon'
    \end{equation}
    for $j = 1, \dots, J$. We assume $\epsilon' < 1$.  It is at this point that the two proofs diverge.

\textit{Second half of the proof of Lemma \ref{lemma:alpha-approx-single}:} 
Recall that here $\eta_1$ satisfies \ref{enum:eta1-1}-\ref{enum:eta1-4} and $\eta_2 \in C^{s'}(\Obar)$ and $\eta_2 \neq 0$ on $\Obar$. Without loss of generality we may assume that $\eta_2 > 0$ on $\Obar$. Since $\Obar$ is compact and $\eta_2$ is continuous and positive, 
\begin{align*}
    \eta_2^- := \min_{x\in \Obar} \eta_2(x) > 0, \quad \eta_2^+ := \max_{x \in \Obar} \eta_2(x) < \infty. 
\end{align*}

We retain $R' = (R_1, \dots, R_J)$ and \eqref{eqn:R-approx} from the first half of the proof, and abbreviate $c_j := \la u, \xi_{j,\delta} \ra$ and $\widehat{c}_j := \int_\Omega R_j(u(y), y)\eta_1(y) \dy$ so that \eqref{eqn:R-approx} reads
$\sup_{u \in \sfK} \max_{j \in [J]} |\widehat{c}_j - c_j| \leq \epsilon'$ with $\epsilon' < 1$. We have that 
\[\max_{j \in [J]} |c_j| \leq \|u_\delta\| \leq \|u\| \leq \sup_{u \in \sfK} \|u\| < \infty.\]
As a result, $\max_{j \in [J]} |\widehat{c}_j| \leq \sup_{u \in \sfK}\|u\| + 1 \eqcolon M_1$ uniformly over $\sfK$. By Lemma \ref{lem:NN-approx}, for any $\epsilon_\eta >0$, there exists a finite-dimensional neural network $R_\eta$ such that 
\begin{equation}
    \|R_\eta - \eta\|_{C^{s'}(\Obar)} < \epsilon_\eta \leq 1.
\end{equation}
Define $M_e = \|\eta\|_{L^\infty(\Obar)} + 1$, which serves as a bound on $\|R_\eta\|_{L^\infty(\Obar)}$. Define the lifting layer of the NNO $\cR: u(x) \mapsto R(u(x), x)$ by 
\begin{equation}
    R(u(x), x) = (R_1(u(x),x), \dots, R_J(u(x),x), 1, R_\eta(x)).
\end{equation}
We draw attention to the fact that the $1$ is easily achieved \textit{exactly} due to the final layer of finite-dimensional neural networks being affine in Definition \ref{def:finite-NN}. Setting the bias term equal to this value of $y$ and the linear weights to $0$ for the channel corresponding to $1$ achieves the value exactly for every input. We set the single hidden layer of our NNO as 
\begin{equation}\label{eqn:L1-single-hidden}
(\cL_1 v)(x) = \sigma\Big(A_1(W_1 v(x) + T_{1,0} \int_{\Omega} v(y) \eta_1(y) \dy\; \eta_2(x)) + b_1\Big).
\end{equation}
We set $T_{1,0}$ to be a diagonal matrix, equal to the identity $I_J$ for channels $1, \dots, J$, equal to $\frac{1}{\|\eta_1\|_{L^1(\Omega)}}$ on channel $J+1$ and equal to $0$ on the final $k'$ entries (corresponding to the $R_\eta$ block). We set $W_1$ to copy the final $k'$ entries ($R_\eta$ block) and $0$ the other entries. The end result of these choices is that before applying $A_1$ and $b_1$ and the activation, we obtain a vector 
\[z(x) := (\widehat{c}_1\eta_2(x), \dots \widehat{c}_J\eta_2(x), \eta_2(x), R_\eta(x)),\]
which for all $u \in \sfK, x \in \Obar$, lies in the box 
\[\cB: [-M_1 \eta_2^+, M_1\eta_2^+]^J \times [\eta_2^-, \eta_2^+] \times [-M_e, M_e]^{k'}.\]
Note that nonnegativity of $\eta_1$ is used to achieve $\eta_2(x)$ in the second-to-last channel of $z$. Define $C_z := \sup_{u \in \sfK}\|z\|_{C^{s'}(\Obar)}$. Set 
\[\tilde{Q}(\{c_j\}_{j=1}^J) = \alpha\Big(\sum_{j=1}^J c_j \xi_j\Big).\]
$\tilde{Q}$ is continuous on $[-M_1, M_1]^J$, so by Lemma \ref{lem:NN-approx}, for any $\epsilon_Q$, there exists a neural network $Q_0$ with 
\begin{equation*}
    \|Q_0 - \tilde{Q}\|_{L^\infty([-M_1, M_1]^J)} \leq \epsilon_Q < 1.
\end{equation*}
 Define the bound $M_{Q_0} = \|Q_0\|_{L^\infty([-M_1, M_1]^J)} \leq \|\tilde{Q}\|_{L^\infty([-M_1, M_1]^J)} + 1$. Define the function $G: \cB \to \R^{k'}$ by 
\[G(\{a_j\}_{j=1}^J, b, e) = Q_0\Bigg( \Big\{\frac{a_j}{b}\Big\}_{j=1}^J\Bigg)e.\]
Since $b \geq \eta_2^- > 0$ on $\cB$, the map $(a_j,b) \mapsto a_j/b$ is $C^\infty$ on that domain, and $\sigma \in C^\infty$ by assumption, so $Q_0 \in C^\infty([-M_1, M_1]^J)$ and $G \in C^\infty(\cB)$. By Lemma \ref{lem:NN-approx}, there exists a finite-dimensional neural network $\cN$ such that 
\begin{equation}
    \|\cN - G\|_{C^{s'}(\cB)} < \epsilon_\cN.
\end{equation}

Set $A_1, b_1$ of $\cL_1$ defined in \eqref{eqn:L1-single-hidden} to be the first layer weights and bias of $\cN$, with $\cQ$ realizing its remaining layers. Then 
$$(\cQ \circ \cL_1 \circ \cR)(u)(x) = \cN(z(x)) \eqcolon \tilde{\alpha}(u)(x).$$
We note that $M_1, \eta_2^-, \eta_2^+, M_e, \cB,$ and $M_{Q_0},$ depend only on $\sfK$, $\eta$, $\eta_2$, $\delta$, $J$, and on the fact that we assume the initial bounds of $\epsilon', \epsilon_\eta, \epsilon_Q$ as being $\leq 1$; they are thus unaffected when these bounds are subsequently decreased below $1$. 

We now estimate the error bound. Note that $z$ is a $C^{s'}(\Obar;\R^{J+1+k'})$ map with components $\widehat{c}_j\eta_2$, $\eta_2$, and $R_\eta \in C^{s'}$. We decompose
\begin{equation*}
    \|\alpha(u)\eta - \tilde{\alpha}(u)\|_{C^{s'}(\Omega)} \leq \underbrace{\|\cN\circ z - G\circ z\|_{C^{s'}(\Omega)}}_{\text{(I)}} + \underbrace{\|G \circ z - \alpha(u) \eta\|_{C^{s'}(\Omega)}}_{\text{(II)}}. 
\end{equation*}
For (I), by multivariate Fa\`a di Bruno's formula, 
\begin{equation*}
    \text{(I)} \leq c_{s',d}(1 + C_z)^{s'}\epsilon_\cN \eqcolon C_\cN\epsilon_\cN,
\end{equation*}
where $c_{s',d}> 0$ is a constant depending on $s'$ and $d$.
For (II), due to the trick of obtaining a channel with $\eta_2(x)$ exactly, the division is exact, and since $\eta_2(x) \geq \eta_2^- >0$, we have that 
\[G(z(x)) = Q_0\left(\Big\{\frac{\widehat{c}_j \eta_2(x)}{\eta_2(x)}\Big\}_{j=1}^J\right)R_\eta(x) = Q_0(\{\widehat{c}_j\})R_\eta,\]
with $Q_0(\{\widehat{c}_j\})$ \textit{independent} of $x$. Then 
\begin{align*}
    \text{(II)} &= \|Q_0(\{\widehat{c}_j\}_{j=1}^J)R_\eta -\alpha(u)\eta\|_{C^{s'}(\Omega)}\\
    & \leq |Q_0(\{\widehat{c}_j)\}_{j=1}^J|\|R_\eta - \eta\|_{C^{s'}(\Omega)} + |Q_0(\{\widehat{c}_j\}_{j=1}^J) - \alpha(u)| \|\eta\|_{C^{s'}(\Omega)}\\
    &\leq M_{Q_0}\epsilon_\eta + |Q_0(\{\widehat{c}_j\}_{j=1}^J) - \alpha(u)| \|\eta\|_{C^{s'}(\Omega)}.
\end{align*}
We have that 
\begin{align*}
    |Q_0(\{\widehat{c}_j\}_{j=1}^J) - \alpha(u)| & \leq |Q_0(\{\widehat{c}_j\}_{j=1}^J) - \tilde{Q}(\{\widehat{c}_j\}_{j=1}^J)| \\ &+ |\tilde{Q}(\{\widehat{c}_j\}_{j=1}^J) - \tilde{Q}(\{c_j\}_{j=1}^J)| + |\tilde{Q}(\{c_j\}_{j=1}^J) - \alpha(u)| \\
    & \leq \epsilon_Q  + |\alpha(\sum_{j=1}^J \widehat{c}_j \xi_j) - \alpha(\sum_{j=1}^J c_j \xi_j)|+ |\alpha_{\delta, J}(u) - \alpha(u)|.
\end{align*}
Let $\omega_\alpha$ be the modulus of continuity of $\alpha$. Taking the supremum over $\sfK$, we have 
\begin{align*}
    \sup_{u \in \sfK} \text{(II)} & \leq M_{Q_0} \epsilon_\eta +  \sup_{u \in \sfK}\Bigg(\epsilon_Q + |\alpha(\sum_{j=1}^J \widehat{c}_j \xi_j) - \alpha(\sum_{j=1}^J c_j \xi_j)| + |\alpha_{\delta, J}(u) - \alpha(u)|\Bigg)\|\eta\|_{C^{s'}(\Omega)}\\
    & \leq M_{Q_0}\epsilon_\eta  + (\epsilon_Q + \omega_\alpha(|\Omega|^{1/2}J^{1/2}\epsilon') + \epsilon_\alpha)\|\eta\|_{C^{s'}(\Omega)},
\end{align*}
where we have used in the final inequality the conversion from the bound $$\|\sum_{j=1}^J (\widehat{c}_j - c_J) \xi_j \|_{L^2(\Omega)} \leq J^{1/2}\epsilon'$$ to obtain the bound 
$$\|\sum_{j=1}^J (\widehat{c}_j - c_J) \xi_j \|_{L^1(\Omega)} \leq |\Omega|^{1/2} J^{1/2}\epsilon',$$
using Cauchy-Schwarz. Collecting (I) and (II), 
\begin{align*}
    \sup_{u \in \sfK} \|\alpha(u)\eta - \tilde{\alpha}(u)\|_{C^{s'}(\Omega)} \leq C_\cN\epsilon_\cN +  M_{Q_0}\epsilon_\eta  + (\epsilon_Q + \omega_\alpha(|\Omega|^{1/2}J^{1/2}\epsilon') + \epsilon_\alpha)\|\eta\|_{C^{s'}(\Omega)}.
\end{align*}
Given $\varepsilon$, first fix $\delta, J$ via \eqref{eqn:alpha-convergence} such that $\epsilon_\alpha \|\eta\|_{C^{s'}(\Omega)} <\varepsilon/5$. The constants $C_\cN$, $M_{Q_0}$, and the domain $\cB$ are then determined. Choosing 
\[
\epsilon_{\mathcal N}<\tfrac{\varepsilon}{5C_{\mathcal N}},\quad
\epsilon_\eta<\tfrac{\varepsilon}{5M_{Q_0}},\quad
\epsilon_Q<\tfrac{\varepsilon}{5\|\eta\|_{C^{s'}(\Omega)}},\quad
\epsilon'\ \text{with}\ \omega_\alpha(|\Omega|^{1/2}\sqrt J\epsilon')<\tfrac{\varepsilon}{5\|\eta\|_{C^{s'}(\Omega)}}
\]
gives $\sup_{u\in\sfK}\|\alpha(u)\eta-\tilde\alpha(u)\|_{C^{s'}(\Omega)}<\varepsilon$.

\textit{Second half of the proof of Lemma \ref{lemma:alpha-approx}:}
    
    We may assume $\eta$ is not identically zero, since if it were we could set all NNO weights to be zero and satisfy the lemma statement trivially. Since $\eta$ and $\psi_{L,m^*}$ are both continuous on $\Omega$, $\eta$ is not identically $0$, and $\psi_{L,m^*}$ equals $0$ only on a set of Lebesgue measure $0$, there exists an open ball $\hat{\Gamma} \subset \Omega$ on which neither $\eta$ nor $\psi_{L,m^*}$ is equal to $0$. It is important to note that this property implies that $\psi_{L,m^*}$ \textit{does not change sign} on $\widehat{\Gamma}$. For some $\epsilon_\gamma >0$, let $\Gamma$ be the open ball in $\hat{\Gamma}$ described in Lemma \ref{lem:Gamma} for $\eta_2$ on $\hat{\Gamma}$, and assume $\Gamma$ has radius $r$. We then have that for some $\gamma$, 
        \begin{equation}\label{eqn:eta2-gamma}
        \|\gamma\eta_2 - \mathbf{1}\|_{L^\infty(\Gamma)} < \epsilon_\gamma.\end{equation}
     For $\delta_1 >0$, let $\tilde{\Gamma}_{\delta_1}$ be the ball that shares a center with $\Gamma$ and has radius $r-\delta_1$; similarly, let $\tilde{\Gamma}_{\delta_1/2}$ be the ball that shares a center with $\Gamma$ and has radius $r- \frac{\delta_1}{2}$. Assume $\delta_1 < \frac{r}{2}$. Let $\Gamma_{\delta_1}$ be the standard $\frac{\delta_1}{2}$-mollification of the indicator function $\mathbf{1}_{\tilde{\Gamma}_{\frac{\delta_1}{2}}}$ (see Definition \ref{def:standard-moll}). Clearly, $\Gamma_{\delta_1} = 1$ inside $\tilde{\Gamma}_{\delta_1}$ and $0$ on $\Omega \setminus \Gamma$; furthermore, $\Gamma_{\delta_1} \in C^{\infty}(\overline{\Omega})$. 
     Since $\Gamma_{\delta_1}$ is a continuous map from $\Obar \to \R $, for any $\epsilon_0 > 0$, by Lemma \ref{lem:NN-approx} there exists a  neural network $\tilde{R}$ such that 
    \begin{equation}\label{eqn:R-tilde-bnd}
        \|\Gamma_{\delta_1}- \tilde{R}\|_{L^\infty(\Obar)} < \epsilon_0.
    \end{equation}
    For convenience, assume $\epsilon_0 < \frac{1}{2}$ such that the image of $\Obar$ under $\tilde{R}$ is contained in $[-\frac{1}{2}, \frac{3}{2}]$.  We note that without loss of generality, we can assume that the $\eta_1$ and $\eta_2$ pair belong to the first hidden layer, since any preceding hidden layers could simply have the weights $T \equiv  0$ and the remaining finite-dimensional neural network form either part of the lifting neural network $R$ or approximate the identity, by assumption on $\sigma$. Finally, set the initial neural network $R$ of the NNO to be the parallelization of $R'$ from \eqref{eqn:R-approx} and $\tilde{R}$. 
    Thus, $R: \R^k \times \Obar\to \R^{J +1}$. Let the first hidden layer of our NNO have the form 
    \begin{equation}\label{eqn:first-layer}
    (\cL_1v)(x) = \sigma\Big(A_1(W_1v(x) +  T_{1,1}\int_{\Omega} v(y) \eta_1 (y) \dy \;\eta_2(x)) + b_1 \Big),
    \end{equation}
    where $W_1 = \text{diag}([\mathbf{0}_J, 1])$, $T_{1,1} = \text{diag}([\gamma\mathbf{1}_J, 0])$, and $A_1$ and $b_1$ are to be defined later. As in the proof of Lemma \ref{lemma:alpha-approx-single}, we abbreviate $c_j := \la u, \xi_{j,\delta} \ra$ and $\widehat{c}_j := \int_\Omega R_j(u(y), y)\eta_1(y) \dy$ so that \eqref{eqn:R-approx} reads
$\sup_{u \in \sfK} \max_{j \in [J]} |\widehat{c}_j - c_j| \leq \epsilon'$ with $\epsilon' < 1$. The end result of the choices for the first layer \eqref{eqn:first-layer} and $R$ is that before applying $A_1$ and $b_1$ and the activation, we obtain a vector 
    \[z(x) = [\widehat{c}_1\gamma \eta_2(x), \dots, \widehat{c}_J\gamma \eta_2(x), \tilde{R}(x)].\]
     Choose $M_1 >0$ such that $\{z(x)|\; u \in \sfK, x \in \Omega \} \subset [-M_1, M_1]^J \times [-\frac{1}{2}, \frac{3}{2}]$, which is possible since $\sup_{y \in \Omega}|u(y)| < \infty$ and $\sup_{y \in \Omega} |\eta_2(y)| < \infty$ by assumption. 

    Similarly, choose $M_2 >0 $ such that the image of the compact set $\sfK \subset L^1(\Omega; \R^k)$ under the mapping    \[
    (u,x) \mapsto \left(c_1\gamma\eta_2(x), \dots, c_J\gamma \eta_2(x) \right)
    \]
    is contained in $[-M_2,M_2]^J$ for all $u \in \sfK$ and $x \in \Omega$.
    Let $M^* = \max\{M_1, M_2\}$. 
    Now define 
\begin{equation}
    \tilde{Q}: \{c_j \}_{j=1}^J\mapsto \alpha\left(\sum_{j=1}^J c_j \xi_j\right),
\end{equation}
where $\{c_j\}_{j=1}^J \in \R^J$. Let $M_\alpha$ be such that the image of $[-M^*, M^*]^J$ under $\tilde{Q}$ is contained in $[-M_{\alpha}, M_{\alpha}]$. Let $\text{id}$ be the identity function, and let 
$\times(a,b):=ab$ denote the scalar multiplication operator between the outputs of two functions $a$ and $b$. By Lemma \ref{lem:NN-approx}, there exists a neural network $Q$ such that for any $\epsilon_2 > 0$, 
\begin{equation}\label{eqn:bnd-Q}
\|Q - \times(\tilde{Q}, \text{id})\|_{L^\infty([-M^*,M^*]^J \times [-\frac{1}{2}, \frac{3}{2}])} < \epsilon_2.
\end{equation}
Let $D$ be the depth of $Q$ and define $M_Q$ such that the image of $[-M^*,M^*]^J \times [-\frac{1}{2}, \frac{3}{2}]$ under $Q$ and under $\times(\tilde{Q}, \text{id})$ is contained in $[-M_Q, M_Q]$.

We may set the weight choices $A_1$ and $b_1$ of $\cL_1$ in \eqref{eqn:first-layer} as well as $\{W_i\}_{i=2}^D$, 
and $\{b_i\}_{i=2}^D$ of hidden layers $\mathcal{L}_2, \dots, \mathcal{L}_D$ so that $\cL_D \circ \dots \circ \cL_2 \circ \sigma(A_1 v +b_1)$ exactly emulates the neural network $Q(v)$ of \eqref{eqn:bnd-Q} for input $v$.
Within these layers, we may simply set $T_{\ell, m} \equiv 0$.

Set the number of hidden layers $L = D+1$. By assumption $\eta$ is one of the basis elements $\phi_{L,m^*}$ for some $m^* \leq M$, and the corresponding $\psi_{L,m^*}$ is such that $\int_{\tilde{\Gamma}_{\delta_1}} \psi_{L, m^*}(y) \dy \neq 0$ for any $\delta_1$ since we chose $\Gamma \subset \hat{\Gamma}$ to be a set on which $\psi_{L,m^*}$ did not equal $0$.
We may then set the last hidden layer $\cL_{L}$ of the NNO as
\begin{equation}
\mathcal{L}_{L}(v) = \sigma\Big(A_{L}\big(\sum_{m=1}^MT_{L, m} \int_\Omega v(y) \psi_{L,m}(y) \dy\; \phi_{L,m}\big) + b_{L}\Big),
\end{equation}
where $T_{L, m} = 0$ for $m \neq m^*$ and $T_{L, m^*} = \frac{1}{\int_{\tilde{\Gamma}_{\delta_1}} \psi_{L, m^*}(y)\dy}$. 
 The next step is to ensure that the set $$B_{m^*} := \{T_{L, m^*} \int_{\Omega} v(y) \psi_{L, m^*}(y) \dy \; \eta(x): \; v(y) \in [-M_Q, M_Q]\text{ for } y \in \Omega, x \in \Omega\}$$ is bounded. We have 
\begin{align}\label{eqn:M_J}
    \sup_{x \in \Omega} |T_{L, m^*} \int_{\Omega} v(y) \psi_{L, m^*}(y)\dy \;\eta(x)| \leq |T_{L, m^*}|M_Q\|\eta\|_{L^\infty(\Omega)} \|\psi_{L,m^*}\|_{L^1(\Omega)}.
\end{align}
The coefficient $T_{L, m^*}$ is bounded for fixed $m^*$, $\delta_1$, and $\Gamma$, and the remaining quantities are bounded by assumption. Thus we deduce there exists $M_{m^*}$ such that $B_{m^*} \subset [-M_{m^*}, M_{m^*}]^{k'}$. By Assumptions \ref{ass:sig-eta}, for any $\epsilon_I$, there exists an identity approximating neural network $\cI$ that preserves $0$ such that 
\begin{equation}\label{eqn:I2}
    \|\cI - \text{id}\|_{C^{s'}([-M_{m^*}, M_{m^*}]^{k'})}< \epsilon_I. 
\end{equation}
Let $\cI$ define the parameters $A_{L}$ and $b_{L}$ as well as those of the output neural network $\cQ$. Denote by $\tilde{\alpha}$ the resulting constructed NNO: 
\begin{equation*}
    \tilde{\alpha}(u) := \cQ \circ \cL_{L} \circ \dots \circ\cL_{1} \circ \cR(u).
\end{equation*}
We now prove the approximation capability of $\tilde{\alpha}$.
\begin{align}
    \sup_{u \in \sfK} \|\alpha(u) \eta - \tilde{\alpha}(u)\|_{C^{s'}(\Omega)} & \leq \sup_{u \in \sfK} \|\alpha(u) \eta - \alpha_{\delta, J}(u) \eta\|_{C^{s'}(\Omega)} + \sup_{u \in \sfK}  \|\alpha_{\delta, J}(u)\eta - \tilde{\alpha}(u) \|_{C^{s'}(\Omega)}\notag\\
    \label{eqn:alpha-eqn1}& \leq \epsilon_\alpha \|\eta\|_{C^{s'}(\Omega)} + \sup_{u \in \sfK}  \|\alpha_{\delta, J}(u)\eta - \tilde{\alpha}(u) \|_{C^{s'}(\Omega)}
\end{align}
by \eqref{eqn:alpha-convergence}. The first term is bounded by assumption on $\eta$. We expand the second term: 
\begin{align*}
    &\sup_{u \in \sfK}  \|\alpha_{\delta, J}(u)\eta - \tilde{\alpha}(u) \|_{C^{s'}(\Omega)} = \sup_{u \in \sfK} \|\alpha_{\delta, J}(u)\eta - \cQ \circ \cL_{L} \circ \dots \circ \cL_{1} \circ R(u)\|_{C^{s'}(\Omega)}.
\end{align*}
Putting together our definitions of $\cQ$, $\cL_{L}, \dots, \cL_1$, and $R$ and letting $\psi := \psi_{L,m^*}$ we have that 
\begin{align}\label{eqn:full-expression}
    \cQ \circ & \cL_{L} \circ \dots  \circ\cL_{1} \circ \cR(u) = \notag\\
    &\cI\big( T_{L,m^*} \int_\Omega Q\big(\gamma\int_\Omega R'(u(y'),y')\eta_1(y') \dy' \; \eta_2(y), \tilde{R}(y)\big) \psi(y) \dy\; \eta \big).
\end{align}
The remainder of the proof simply involves assembling the individual bounds on each component one at a time. Note that we have already verified that the argument of each neural network -- specifically, the arguments of each of $\cI$, $Q$, $R'$, and $\tilde{R}$ -- falls within the correct domain of approximation. For visual clarity in the remainder of the bound, let $\mathbf{a}_f$ denote the argument of function $f$ in the right-hand side of \eqref{eqn:full-expression}; e.g.~$\mathbf{a}_{R'}$ denotes $u(y'), y'$. Then 
\begin{align}\label{eqn:bnd-begin}
    \sup_{u \in \sfK} \|\alpha_{\delta, J}(u) \eta - \tilde{\alpha}(u) \|_{C^{s'}(\Omega)} & \leq \underbrace{\sup_{u \in \sfK} \|\alpha_{\delta, J}(u)\eta - \mathbf{a}_{\cI}\|_{C^{s'}(\Omega)}}_{\text{(I)}} + \sup_{u \in \sfK} \|\cI(\mathbf{a}_{\cI}) - \mathbf{a}_{\cI} \|_{C^{s'}(\Omega)}.
\end{align}
For the second term, 
\begin{align}
    \sup_{u \in \sfK} \|\cI (\mathbf{a}_{\cI}) - \mathbf{a}_{\cI}\|_{C^{s'}(\Omega)} &\leq c_{s',d} \epsilon_I (1 + \|\mathbf{a}_{\cI}\|_{C^{s'}(\Omega)})\notag\\
    & \leq c_{s',d} \epsilon_I (1 + |T_{L, m^*}| \|\psi\|_{L^1(\Omega)}M_Q\|\eta\|_{C^{s'}(\Omega)})\notag\\
    & \leq C_1 \epsilon_I\label{bnd:C1}
\end{align}
by Fa\`a di Bruno's formula, \eqref{eqn:M_J}, and \eqref{eqn:I2}, where \newline $C_1 = c_{s', d}(1 + C_{\Gamma, \psi}\|\psi\|_{L^1(\Omega)}M_Q \|\eta\|_{C^{s'}(\Omega)})$, and  $C_{\Gamma, \psi} = \sup_{0 < \delta_1 < \frac{r}{2}} \left|\frac{1}{\int_{\tilde{\Gamma}_{\delta_1}} \psi(z) \dz} \right|< \infty$, the bound on $|T_{L,m^*}|$, holds since $\psi \neq 0$ does not change sign on $\Gamma$ and is continuous and bounded on $\Omega$.
Critically, we may take $\delta_1 \to 0$ without affecting $C_{\Gamma, \psi}$.
For (I), 
\begin{align}
    \sup_{u \in \sfK} &\|\alpha_{\delta, J}(u) \eta - \mathbf{a}_{\cI}\|_{C^{s'}(\Omega)}\notag \\& \leq \sup_{u \in \sfK} \left|\alpha_{\delta, J}(u) -  T_{L,m^*} \int_\Omega Q(\mathbf{a}_{Q}) \psi(y) \dy\right| \|\eta\|_{C^{s'}(\Omega)}.\label{eqn:bnd-on-I}
\end{align}
The quantity $\|\eta\|_{C^{s'}(\Omega)}$ is bounded by assumption. Bounding the first factor, 
\begin{align*}
    \sup_{u \in \sfK}& \left|\alpha_{\delta, J}(u) -  T_{L,m^*} \int_\Omega Q(\mathbf{a}_{Q}) \psi(y) \dy\right| \leq\\
    &\underbrace{\sup_{u \in \sfK}\Big|\alpha_{\delta, J}(u) - T_{L, m^*} \int_\Omega \alpha\left(\sum_{j=1}^J \widehat{c}_j\gamma \eta_2(y) \xi_j\right)\tilde{R}(y) \psi(y) \dy \Big|}_{\text{(II)}}\\ & + |T_{L, m^*}| \sup_{u \in \sfK} \Big| \int_\Omega \alpha\left(\sum_{j=1}^J \widehat{c}_j\gamma \eta_2(y) \xi_j\right)\tilde{R}(y) \psi(y) \dy - \int_\Omega Q(\mathbf{a}_Q)\psi(y) \dy\Big|.
\end{align*}
 The second term is bounded by 
\begin{equation}\label{bnd:C2}
\epsilon_2|T_{L,m^*}|  \|\psi\|_{L^1(\Omega)} = C_2\epsilon_2
\end{equation}
by equation \eqref{eqn:bnd-Q}, where $C_2 = C_{\Gamma, \psi}\|\psi\|_{L^1(\Omega)}$.
Continuing with the bound on (II):
\begin{align}
    \sup_{u \in \sfK}&\Big|\alpha_{\delta, J}(u) - T_{L, m^*} \int_\Omega \alpha\left(\sum_{j=1}^J \widehat{c}_j\gamma \eta_2(y) \xi_j\right)\tilde{R}(y) \psi(y) \dy \Big|\notag\\
    & \leq \sup_{u \in \sfK}\Big|\alpha_{\delta, J}(u) - T_{L, m^*} \int_\Omega \alpha\left(\sum_{j=1}^J \widehat{c}_j\gamma \eta_2(y) \xi_j\right)\Gamma_{\delta_1}(y) \psi(y) \dy\Big| \notag\\
    & + |T_{L,m^*}| \epsilon_0 \|\psi\|_{L^1(\Omega)}M_{\alpha},\label{bnd:C3}
\end{align}
by \eqref{eqn:R-tilde-bnd}.
Let $C_3 = C_{\Gamma, \psi}\|\psi\|_{L^1(\Omega)}M_\alpha$. Continuing to bound the first term in line \eqref{bnd:C3}, 
\begin{align*}
     &\sup_{u \in \sfK}\Big| \alpha_{\delta, J}(u) - T_{L, m^*}\int_\Omega \alpha\left(\sum_{j=1}^J \widehat{c}_j\gamma \eta_2(y) \xi_j\right)\Gamma_{\delta_1}(y) \psi(y) \dy \Big| =\\
     \leq 
      &\underbrace{\sup_{u \in \sfK}\Big| \alpha_{\delta, J}(u) - T_{L, m^*}\int_{\Omega} \alpha\Big(\sum_{j=1}^J  \widehat{c}_j\;  \xi_j \Big)\Gamma_{\delta_1}(y)\psi(y) \dy \Big|}_{\text{(III)}}   \\
      &+ \sup_{u \in \sfK} \Big|T_{L, m^*} \int_{\Omega}  \alpha\Big(\sum_{j=1}^J \widehat{c}_j\; \xi_j \Big)\Gamma_{\delta_1}(y)\psi(y) \dy \\ 
      & - T_{L, m^*} \int_{\Omega}  \alpha\Big(\sum_{j=1}^J\widehat{c}_j \; \gamma \eta_2(y) \xi_j \Big)\Gamma_{\delta_1}(y)\psi(y) \dy \Big|.
      \end{align*}
      The second supremum term is bounded by 
      \begin{align}\label{eqn:second-sup-term}
       |T_{L, m^*}|\|\psi\|_{L^1(\Gamma)}\sup_{u \in \sfK}\sup_{y \in \Gamma} \big|\alpha\big(\sum_{j=1}^J\widehat{c}_j \; \xi_j\big) - \alpha\big(\sum_{j=1}^J \widehat{c}_j\; \gamma\eta_2(y) \xi_j\big) \big|. 
\end{align}
Let $\omega_\alpha$ be the modulus of continuity of $\alpha$. From the orthonormality of $\xi_j$ in $L^2(\Omega)$ and H\"older's inequality, it holds that 
\[\|\sum_{j=1}^J \widehat{c}_j(1-\gamma\eta_2(y))\xi_j\|_{L^1(\Omega)} \leq |\Omega|^{1/2} \sum_{j=1}^J |\widehat{c}_j ||1- \gamma\eta_2(y)|,\] and using \eqref{eqn:R-approx} and \eqref{eqn:eta2-gamma}, we may then bound \eqref{eqn:second-sup-term} by
\begin{equation}\label{bnd:C35}
    |T_{L, m^*}|\|\psi\|_{L^1(\Gamma})\sup_{u \in \sfK}\omega_\alpha(\epsilon_\gamma |\Omega|^{1/2}J \max_{j \in [J]} (| \la u, \xi_{j, \delta}\ra| + 1)) \leq |T_{L, m^*}|\|\psi\|_{L^1(\Gamma)}\omega_{\alpha}(\epsilon_\gamma C_{\Omega, J, \delta}).
\end{equation}

Next, we show that 
$$|T_{L,m^*}|\|\psi\|_{L^1(\Gamma)} = \frac{\int_\Gamma |\psi(y)| \dy }{|\int_{\tilde{\Gamma}_{\delta_1}} \psi(y) \dy|}$$ 
may be bounded independently of $\Gamma$ given fixed $\widehat{\Gamma}$ through sufficiently small choice of $\delta_1$. Since $\psi$ does not change sign on $\widehat{\Gamma}$, write $|\int_{\tilde{\Gamma}_{\delta_1}} \psi(y) \dy| = \int_{\tilde{\Gamma}_{\delta_1}} |\psi(y)|\dy = \int_\Gamma |\psi(y)| \dy - \int_{\Gamma \setminus \tilde{\Gamma}_{\delta_1}} |\psi(y)| \dy \geq \int_\Gamma |\psi(y)| \dy - \|\psi\|_{L^\infty(\Omega)} d\omega_d r^{d-1} \delta_1$, where $\omega_d$ is the volume of the unit ball in $\Rd$, and we have used the fact that the volume of the shell $\Gamma \setminus \tilde{\Gamma}_{\delta_1}$ is upperbounded by $d\omega_d r^{d-1} \delta_1$. Since $\psi$ is continuous and does not equal $0$ on $\Gamma$, $\int_\Gamma |\psi(y)| \dy>0$, so that $\delta_1^*(\Gamma) \coloneq \frac{\int_\Gamma |\psi(y)|\dy }{2 \|\psi\|_{L^\infty(\Omega)} d \omega_d r^{d-1}}$. Then for any $\delta_1$ such that $0< \delta_1 \leq \delta_1^*(\Gamma)$, we have $|T_{L,m^*}|\|\psi\|_{L^1(\Gamma)} \leq 2$, so \eqref{bnd:C35} becomes simply 
\begin{equation}\label{bnd:C355}
    |T_{L, m^*}|\|\psi\|_{L^1(\Gamma)}\sup_{u \in \sfK}\omega_\alpha(\epsilon_\gamma |\Omega|^{1/2}J \max_{j \in [J]} (| \la u, \xi_{j, \delta}\ra| + 1)) \leq 2\omega_{\alpha}(\epsilon_\gamma C_{\Omega, J, \delta}),
\end{equation}
so long as we choose $\delta_1 < \delta_1^*(\Gamma)$ after choosing $\Gamma$.
Now having dealt with the approximation of $1$ by $\gamma\eta_2$ on $\Gamma$, we continue to bound term (III):
\begin{align}
    &\sup_{u \in \sfK}\Big| \alpha_{\delta, J}(u) - T_{L, m^*}\int_{\Omega} \alpha\Big(\sum_{j=1}^J  \int_\Omega R_j(\mathbf{a}_{R_j})\eta_1(y') \dy'\;  \xi_j \Big)\Gamma_{\delta_1}(y)\psi(y) \dy \Big |\notag\\
     \leq &\underbrace{\sup_{u \in \sfK}\Big| \alpha_{\delta, J}(u) - T_{L, m^*}\int_{\Omega} \alpha\Big(\sum_{j=1}^J  \la u, \xi_{j,\delta}\ra\;  \xi_j \Big)\Gamma_{\delta_1}(y)\psi(y) \dy \Big |}_{\text{(IV)}} \notag\\
     & +  \sup_{u \in \sfK}\Big|T_{L, m^*}\int_{\Omega} \alpha\Big(\sum_{j=1}^J  \int_\Omega R_j(\mathbf{a}_{R_j})\eta_1(y') \dy'\;  \xi_j \Big)\Gamma_{\delta_1}(y)\psi(y) \dy \notag\\
     &- T_{L, m^*}\int_{\Omega} \alpha\Big(\sum_{j=1}^J  \la u, \xi_{j,\delta}\ra\;  \xi_j \Big)\Gamma_{\delta_1}(y)\psi(y) \dy \Big |\notag\\
     & \leq \text{(IV)}+
     |T_{L,m^*}|\|\psi\|_{L^1(\Omega)}\omega_{\alpha}(J^{1/2}\epsilon'|\Omega|^{1/2}) \label{bnd:C4}
    \end{align}
by \eqref{eqn:R-approx} and another application of H\"older's inequality within the argument of $\omega_\alpha$.
Set $C_4 = C_{\Gamma, \psi}\|\psi\|_{L^1(\Omega)}$. We can see that, within (IV), 
\begin{align*}
    &T_{L, m^*} \int_\Omega \alpha\left(\sum_{j=1}^J \la u, \xi_{j, \delta} \ra \xi_j\right)\Gamma_{\delta_1}(y) \psi(y) \dy \\
    & = \alpha\left(\sum_{j=1}^J \la u, \xi_{j, \delta} \ra \xi_j\right)\Big(1 + \int_{\Gamma \setminus \tilde{\Gamma}_{\delta_1}} \frac{1}{\int_{\tilde{\Gamma}_{\delta_1}} \psi(z) \dz } \psi(y) \Gamma_{\delta_1}(y) \dy \Big).
\end{align*}
Since $\alpha(\sum_{j=1}^J \la u, \xi_{j, \delta} \ra \xi_j) = \alpha_{\delta, J}(u)$ by definition, so finally we have that 
\begin{align}
    \text{(IV)} & \leq \sup_{u \in \sfK} |\alpha_{\delta, J}(u)| \left| \int_{\Gamma \setminus \tilde{\Gamma}_{\delta_1}} \frac{1}{\int_{\tilde{\Gamma}_{\delta_1}} \psi(z) \dz } \psi(y) \Gamma_{\delta_1}(y) \dy\right|\notag\\
    & \leq M_{\delta,J}\|\psi\|_{L^\infty(\Omega)}C_{\Gamma,\psi}\omega_d d r^{d-1} \delta_1,\label{bnd:C5}
\end{align}
where $M_{\delta,J} = \sup_{u \in \sfK} |\alpha_{\delta,J}(u)| < \infty$ by continuity of $\alpha_{\delta,J}$ and compactness of $\sfK$, $\omega_d$ is the volume of the unit ball in $\Rd$, and we have used the fact that the volume of the shell $\Gamma \setminus \tilde{\Gamma}_{\delta_1}$ is upperbounded by $d\omega_d r^{d-1} \delta_1$. Let $C_5 = M_{\delta,J}\|\psi\|_{L^\infty(\Omega)} C_{\Gamma, \psi}\omega_d dr^{d-1}$.

Bringing together the bounds of all the terms we have obtained, from equations \eqref{eqn:bnd-begin}, \eqref{bnd:C1}, \eqref{eqn:bnd-on-I}, \eqref{bnd:C2}, \eqref{bnd:C3}, \eqref{bnd:C355}, \eqref{bnd:C4}, and \eqref{bnd:C5}, we have
\begin{align*}
    \sup_{u \in \sfK}&\|\alpha_{\delta, J}(u) \eta - \tilde{\alpha}(u)\|_{C^{s'}(\Omega)} \leq C_1\epsilon_I + \|\eta\|_{C^{s'}(\Omega)}\Big(C_2\epsilon_2 + C_3\epsilon_0 \\ &+ 2 \omega_\alpha(\epsilon_\gamma C_{\Omega, J, \delta}) + C_4 \omega_\alpha(J^{1/2}\epsilon'|\Omega|^{1/2}) + C_5\delta_1\Big)
\end{align*}
The choices of $\Omega$, $\delta$, $s'$, $\eta$, $m^*$, $\psi$, $\widehat{\Gamma}$, and $J$ are fixed. For any $\varepsilon' >0$, we then may choose $\epsilon_\gamma$ such that $\omega_\alpha(\epsilon_{\gamma} C_{\Omega, J, \delta}) \leq \frac{\varepsilon'}{14\|\eta\|_{C^{s'}(\Omega)}}$. The choice of $\epsilon_\gamma$ determines $\Gamma$, $\gamma$, and $r$, which determines the value of $C_{\Gamma, \psi}$ as well as the domain bounds $M^*$, $M_\alpha$, and $M_Q$. Then constants $C_1$, $C_2$, $C_3$, $C_4$, and $C_5$ are determined. Choose $\epsilon '$ such that $\omega_{\alpha}(J^{1/2}\epsilon' |\Omega|^{1/2})\leq \frac{\varepsilon'}{7C_4\|\eta\|_{C^{s'}(\Omega)}}$.  We choose $\delta_1$ such that $\delta_1\leq \min\left(\frac{\varepsilon'}{7\|\eta\|_{C^{s'}(\Omega)}C_5}, \frac{r}{2}, \delta_1^*(\Gamma)\right)$, $\epsilon_2$ such that $\epsilon_2 < \frac{\varepsilon'}{7C_2\|\eta\|_{C^{s'}(\Omega)}}$, $\epsilon_0$ such that $\epsilon_0 < \frac{\varepsilon'}{7C_3\|\eta\|_{C^{s'}(\Omega)}}$, and $\epsilon_I$ such that $\epsilon_I < \frac{\varepsilon'}{7C_1}$. 
Combining this result with \eqref{eqn:alpha-convergence}, we have
\begin{align*}
    \sup_{u \in \sfK} \|\alpha(u) \eta - \tilde{\alpha}(u) \|_{C^{s'}(\Omega)} & \leq \sup_{u \in \sfK} \|\alpha(u) \eta - \alpha_{\delta, J}(u) \eta \|_{C^{s'}(\Omega)} + \sup_{u \in \sfK}\|\alpha_{\delta, J}\eta - \tilde{\alpha}(u) \|_{C^{s'}(\Omega)} \\
    & \leq \sup_{u \in \sfK} |\alpha(u) - \alpha_{\delta, J}(u) | \|\eta\|_{C^{s'}(\Omega)} + \varepsilon'\\
    & \leq \epsilon_{\alpha}\|\eta\|_{C^{s'}(\Omega)} + \varepsilon'.
\end{align*}
Choosing $\varepsilon' < \frac{\varepsilon}{2}$ and $\epsilon_\alpha \leq \frac{\varepsilon}{2\|\eta\|_{C^{s'}(\Omega)}}$, we have the desired result \eqref{eqn:lemma-alpha-approx}.
    
\end{proof}

\section{Eigenfunctions of the Laplacian}
\label{sec:eigfunc-Lap}


Here we state two lemmas related to approximation by sum of nonlinear functionals and eigenfunctions of Laplacian. The first lemma gives a convergence result for projection onto the basis of the eigenfunctions of the Laplacian. Though results of this type are generally well-known \cite{evans2022partial, gilbarg1998elliptic}, we were unable to find a proof for this particular case in the literature, so we have included a proof here. We note that the compact support assumption of the set to be approximated is essential since, by construction, projection onto the eigenfunctions of the Laplacian nullifies at the boundary the projected function $w$ and every power $\Delta^k w$ for integer $k \geq 0$.

\begin{lemma}[Projection onto the Eigenfunctions of the Laplacian]\label{lem:eigfunc-proj}
    Let $\Omega$ be a bounded Lipschitz domain in $\Rd$. Let $W$ be a bounded set in 
$C^k_c(\Obar)$, which is the subspace of $C^k(\Obar)$ whose functions have compact support in $\Omega$. Assume $k>d+1$ is an integer. Let $\{\lambda_1, \eta_1\}, \{\lambda_2, \eta_2\}, \dots$ be the eigenpairs of the Laplacian on $\Omega$ subject to homogeneous Dirichlet boundary conditions, ordered such that $0 < \lambda_1 \leq \lambda_2, \cdots$. Let $P_Jw = \sum_{j=1}^J \la w, \eta_j\ra_{L^2(\Omega)} \eta_j$ be the projection of $w$ onto the first $J$ eigenfunctions over $\Omega$. Then the approximation via the projection $P_J$ converges uniformly in $C(\Obar)$  as $J \to \infty$:
    \begin{equation*}
        \lim_{J \to \infty} \sup_{w \in W} \Big\|w - \sum_{j=1}^J \la w, \eta_j \ra_{L^2(\Omega)}\eta_j\Big\|_{C(\Obar)}=0.
    \end{equation*}
\end{lemma}
\begin{proof}
    Let $e_J = w - P_J w = \sum_{j = J+1}^\infty \la w, \eta_j\ra_{L^2(\Omega)} \eta_j$ since $\{\eta_j\}_{j=1}^\infty$ form a basis for $L^2(\Omega)$. We recall that $\eta_j \in C^\infty(\Omega) \cap C(\Obar)$, and $\eta_j = 0$ on $\p\Omega$ \cite{evans2022partial}.

    We bound each inner product individually via 
    \begin{align*}
        \la w, \eta_j \ra & = \int_{\Omega} w \eta_j \dx \\
        & = -\int_{\Omega} \frac{1}{\lambda_j} w \Delta \eta_j \dx \\
        & =  \frac{1}{\lambda_j}\int_\Omega \nabla w \cdot \nabla \eta_j \dx - \frac{1}{\lambda_j}\int_{\p \Omega} w(\nabla \eta_j \cdot \hat{n}) \dx\\
        & = \frac{1}{\lambda_j}\int_{\partial \Omega} \eta_j (\nabla w \cdot \hat{n}) \dx - \frac{1}{\lambda_j} \int_{\Omega} \eta_j \nabla\cdot(\nabla w) \dx \\
        & =  -\frac{1}{\lambda_j}\int_\Omega \eta_j \Delta w \dx.
    \end{align*}
    Since $w$ and all its derivatives have compact support in $\Omega$, we may iterate this procedure $r$ times to obtain 
    \begin{equation}\label{eqn:eigfunc-IP-bnd}
    \la w, \eta_j \ra  = \frac{(-1)^r}{\lambda_j^r} \la \Delta^r w, \eta_j \ra.
\end{equation}

Furthermore, from \cite[Lemma 2.1]{davies1984ultracontractivity} combined with the unlabeled bound on $a_t(x,y)$ contained on \cite[p. 377]{davies1984ultracontractivity}, we can bound the eigenfunctions via 
\begin{equation*}
    \|\eta_j\|_{L^\infty(\Obar)} \leq e^{(\lambda_j/2)t}(4\pi t)^{-d/4}
\end{equation*}
for any $t>0$. Optimizing over $t$, we have the bound 
\begin{equation}
    \|\eta_j\|_{L^\infty(\Obar)} \leq c_d \lambda_j^{d/4},
\end{equation}
where $c_d$ is a constant dependent only on $d$. Then we may bound the tail
\begin{align*}
    \|e_J\|_{C(\Obar)} &\leq \sum_{j > J} |\la w, \eta_j \ra | \|\eta_j\|_{L^\infty(\Obar)}\\
    & \leq c_d \sum_{j > J}\lambda_j^{d/4- r} | \la \Delta^r w, \eta_j \ra | \\
    & \leq c_d \big(\sum_{j > J} \lambda_j^{d/2 - 2r}\big)^{1/2} \big(\sum_{j > J} |\la \Delta^r w, \eta_j \ra|^2\big)^{1/2}\\
    & \leq c_d \|\Delta^r w\| \big(\sum_{j > J} \lambda_j^{d/2 - 2r}\big)^{1/2}.
\end{align*}

We now must bound the summation. Weyl's law states that the number $N(\lambda)$ of these eigenvalues that are less than or equal to $\lambda$ satisfies 
    \[
    \lim_{\lambda \to \infty} \frac{N(\lambda)}{\lambda^{d/2}} = (2\pi)^{-d} \omega_d \text{vol}(\Omega),
    \]
    where $\omega_d$ is the volume of the unit ball in $\Rd$ \cite{netrusov2005weyl, weyl1912asymptotische}. Thus, for any choice of $\epsilon$, there exists a $\lambda_{\epsilon}$ such that for all $\lambda > \lambda_{\epsilon}$, 
    \[
    \Big|\frac{N(\lambda)}{\lambda^{d/2}} - (2\pi)^{-d}\omega_d \text{vol}(\Omega)\Big| < \epsilon.
    \]
    To handle duplicate eigenvalues, let $j^*$ be the largest index $k$ such that $\lambda_k \leq \lambda_{j}$. Note that if $\lambda_j$ is not a duplicate eigenvalue, then $j^* = j$. Then for $\lambda_j > \lambda_{\epsilon}$,
    \[
    \Big|\frac{j^*}{\lambda_j^{d/2}} - (2\pi)^{-d}\omega_d \text{vol}(\Omega)\Big| < \epsilon.
    \]
    Rearranging yields
    \[
    \frac{1}{\lambda_j} < \left(\frac{\epsilon + (2\pi)^{-d}\omega_d \text{vol}(\Omega)}{j^*}\right)^{2/d},
    \]
    which holds for $j > J_{\epsilon}$, where $J_{\epsilon}$ is the smallest index such that $\lambda_{J_{\epsilon}} > \lambda_{\epsilon}$.

    Then, 
    \begin{align*}
        \|e_{J_{\epsilon}}\|_{C(\Obar)} &\leq c_d\|\Delta^r w\| \left(\sum_{j > J_\epsilon} \left(\frac{\epsilon + (2\pi)^{-d}\omega_d \text{vol}(\Omega)}{j}\right)^{\frac{4r}{d} - 1}\right)^{1/2} \\
        & \leq c_d\|\Delta^r w\|(\epsilon + (2\pi)^{-d}\omega_d \text{vol}(\Omega))^{\frac{2r}{d} - \frac{1}{2}} \left(\int_{J_\epsilon}^\infty \frac{1}{j^{\frac{4r}{d} -1}}\mathsf{d} j\right)^{1/2}
    \end{align*}

    For convergence, we require $r > \frac{d}{2}$; let $r$ be the smallest such integer that satisfies this constraint. So that $\Delta^r w$ is defined and in $L^2$, we require $2r \leq k$. By assumption, $W \subset C^k_c(\Obar)$ for $k > d+1$, so $\sup_{w \in W} \|\Delta^r w\|$ is finite for this choice of $r$. The bound becomes 
    \[ \|e_{J}\|_{C(\Obar)} \leq c_{d,W,r,\Omega} J^{1-\frac{2r}{d}},\]
    for $J > J_1$, where $c_{d,W,r,\Omega} = \sup_{w \in W} \|\Delta^{r} w\|(1 + (2\pi)^{-d}\omega_d \text{vol}(\Omega))^{\frac{2r}{d} - \frac{1}{2}}$. We have set $\epsilon = 1$, as it plays no role in the convergence. Since this is the tail bound, we have that 
       \begin{equation*}
        \lim_{J \to \infty} \sup_{w \in W} \Big\|w - \sum_{j=1}^J \la w, \eta_j \ra_{L^2(\Omega)}\eta_j\Big\|_{C(\Obar)} = \lim_{J \to \infty} \sup_{w \in W} \|e_{J}\|_{C(\Obar)}=0.
    \end{equation*}
\end{proof}

\end{document}